\documentclass[10pt, a4paper]{amsart}
 
\usepackage[T1]{fontenc}
\usepackage{microtype}

\usepackage{amsmath}								
\usepackage{amssymb}
\usepackage{amsthm}
\usepackage{amscd}
\usepackage{amsfonts}
\usepackage{bbm}
\usepackage{stmaryrd}

\usepackage{extarrows}

\usepackage{color}
\usepackage{xcolor}
\usepackage{comment}
\usepackage{mathrsfs}
\usepackage{bm}
\definecolor{GoetheBlue}{RGB}{0,97,143}
\usepackage[bookmarks=true, colorlinks=true, linkcolor=blue!50!black,
citecolor=GoetheBlue, urlcolor=GoetheBlue, pdfencoding=unicode]{hyperref}

\usepackage{tikz}									
\usetikzlibrary{matrix}
\usetikzlibrary{patterns}
\usetikzlibrary{matrix}
\usetikzlibrary{positioning}
\usetikzlibrary{decorations.pathmorphing}
\usetikzlibrary{cd}
\usetikzlibrary{trees}
\usepackage{hyperref}
\usepackage{caption}

\usepackage{fullpage}
\usepackage{enumerate}
\usepackage{mathtools}
\usepackage{parskip}
\usepackage[new]{old-arrows}
\usepackage[left=3.5cm,right=3.5cm,top=4cm,bottom=4cm,footskip=1cm
]{geometry}
\usepackage{multicol}

\makeatletter
\def\thm@space@setup{%
  \thm@preskip=5pt \thm@postskip=5pt
}
\makeatother

\newtheorem*{theorem*}{Theorem}

\newtheorem{maintheorem}{Theorem}[section]

\newtheorem{theorem}{Theorem}[section]
\newtheorem{lemma}[theorem]{Lemma}
\newtheorem{proposition}[theorem]{Proposition}
\newtheorem{corollary}[theorem]{Corollary} 

\theoremstyle{definition}
\newtheorem{definition}[theorem]{Definition}
\newtheorem{example}[theorem]{Example}

\newtheorem{remark}[theorem]{Remark}

\newtheoremstyle{myitemstyle}						
	{}			
	{}			
	{}			
	{}			
	{}			
	{.}			
	{ }			
	{}			
\theoremstyle{myitemstyle}
\newtheorem{myitemthm}{}

\newcommand{\R}{\mathbb{R}}
\newcommand{\HH}{\mathbb{H}}

\newcommand{\Z}{\mathbb{Z}}

\newcommand{\Q}{\mathbb{Q}}
\newcommand{\C}{\mathbb{C}}

\newcommand{\N}{\mathbb{N}}
\renewcommand{\P}{\mathbb{P}}
\newcommand{\PP}{\mathbb{P}}

\newcommand{\A}{\mathbb{A}}

\newcommand{\TT}{\mathbb{T}}
\newcommand{\T}{\mathbb{T}}

\newcommand{\RX}{\mathfrak{R}\mathcal{X}}
\newcommand{\RP}{\mathbb{\R}\mathbb{P}}

\newcommand{\calBbar}{\overline{\mathcal{B}}}

\newcommand{\calC}{\mathcal{C}}

\newcommand{\calM}{\mathcal{M}}

\newcommand{\calX}{\mathcal{X}}
\newcommand{\calXbar}{\overline{\mathcal{X}}}

\newcommand{\Muniv}{\mathcal{M}_{\univ}}
\newcommand{\RTP}{\mathfrak{R}\T\P}

\DeclareMathOperator{\Spec}{Spec}
\DeclareMathOperator{\Proj}{Proj}
\DeclareMathOperator{\Hom}{Hom}

\DeclareMathOperator{\val}{val}

\DeclareMathOperator{\supp}{supp}
\DeclareMathOperator{\Trop}{Trop}

\DeclareMathOperator{\id}{id}

\DeclareMathOperator{\Span}{Span}

\DeclareMathOperator{\PGL}{PGL}
\DeclareMathOperator{\trop}{trop}

\DeclareMathOperator{\ev}{ev}
\DeclareMathOperator{\an}{an}
\DeclareMathOperator{\univ}{univ}

\DeclareMathOperator{\sgn}{sgn}
\DeclareMathOperator{\Supp}{Supp}
\DeclareMathOperator{\Cov}{Cov}

\DeclareMathOperator{\triv}{triv}

\newcommand{\siinorm}[1]{\left\lVert #1  \right \rVert^{\text{sgn}}}
\newcommand{\sinorm}[1]{\left\lvert #1  \right \rvert^{\text{sgn}}}

\title{The Signed Goldman--Iwahori space\\
and Real Tropical Linear Spaces} 

\date{}

\author{Kevin K\"uhn}
\address{Technische Universit\"at Berlin,
10587 Berlin, Germany}
\email{kuehn@math.tu-berlin.de}

\author{Arne Kuhrs}
\address{Max Planck Institute for Mathematics in the Sciences,
04103 Leipzig, Germany}
\email{arne.kuhrs@mis.mpg.de}

\begin{document}

\noindent
\maketitle

\begin{abstract} 
For a real closed field $K$ with a non-Archimedean absolute value, we introduce the
\emph{signed Goldman--Iwahori space}, the space of homothety classes of signed
seminorms on a finite-dimensional vector space over $K$. This space merges two
geometric perspectives: it is the linear-algebraic analogue of the real
analytification of projective space introduced by Jell, Scheiderer, and Yu, and a
signed refinement of the Goldman--Iwahori space of seminorms studied in our previous
work. Our first main result identifies it as the inverse limit of all real
tropicalizations of projective space, the real analogue of a theorem of Payne. Our
second main result gives a matroid-theoretic description in terms of the universal
realizable oriented valuated matroid on the underlying vector space.
In the constant coefficient case for $K = \R$, signed seminorms are exactly the diagonalizable ones which is a consequence of hyperplane separation that fails for all other real closed fields. In this case, this yields an explicit description of the signed Goldman--Iwahori space in terms of signed flags of subspaces and lets us show that the space coincides with the real Bergman fan of the universal oriented matroid.
\end{abstract}

{
\renewcommand{\baselinestretch}{0.75}\normalsize
\hypersetup{hidelinks}
\setcounter{tocdepth}{1}
\tableofcontents
\renewcommand{\baselinestretch}{1.0}\normalsize
}

\section*{Introduction}
\thispagestyle{empty}

The Goldman--Iwahori space is the space of non-Archimedean seminorms on a finite-di\-men\-sio\-nal vector space over a non-Archimedean field. It was introduced in \cite{GoldmanIwahori} to give a non-Archimedean analogue of symmetric spaces for the group $\PGL$. If the ground field is spherically complete, this space is a compactification of a building in the sense of Bruhat and Tits. Bruhat--Tits buildings have proven to be an effective framework to study reductive groups over non-Archimedean fields (see \cite{Bruhat_Tits_I, Bruhat_Tits_II}). Their close relation to Berkovich geometry is well-established in the literature (see \cite{berkovich, Werner_seminorms, RemyThuillierWernerI, RemyThuillierWerner_survey}).  Expanding on \cite{DressTerhalle_treeoflife}, the authors in \cite{BKKUV} introduced a tropical approach to study the Goldman--Iwahori space by identifying it with the limit of tropicalized linear spaces and giving a matroidal interpretation.

Suppose the ground field is now also real closed such that the unique ordering is compatible with the absolute value. This paper presents a signed analogue of the Goldman--Iwahori space, taking the order of the field into consideration. We approach this \emph{signed} \emph{Goldman--Iwahori space} by methods of real tropical geometry. This extends the non-Archimedean framework to real tropical geometry, which was introduced in \cite{JSY22}.

We recall some basics of tropical geometry: let $K$ be an algebraically closed field with a non-trivial non-Archimedean absolute value $\lvert\cdot \rvert$. We denote the \emph{tropicalization map} by 
$$\trop: K^n \to (\R\cup \{\infty\})^n,\quad (x_1,\dots,x_n) \mapsto (-\log |x_1|,\dots,-\log|x_n|)\, .$$
The closure of the image of an algebraic subvariety $X\subseteq K^n$ under $\trop$ is called the \emph{tropicalization} of $X$ and usually denoted by $\Trop(X)$. This space has the structure of a finite polyhedral complex of dimension $\dim X$ and its combinatorics carries rich information about the variety $X$, like its degree and Chow cohomology class. 

It is a fundamental trait of tropical geometry that $\Trop(X)$ not only depends on the variety $X$, but also on the chosen embedding into $K^n$. Hence, we will write the tropicalization as $\Trop(X,\iota)$ to emphasize the dependence on the embedding $\iota: X\hookrightarrow K^n$. More generally, one can define tropicalizations of closed subvarieties of any toric variety, not just $K^n$ \cite[\S 3]{payne_analytification}. One can also drop the conditions of $K$ being algebraically closed or non-trivially valued by choosing an algebraically closed, non-trivially field valued extension $L/K$ and defining $\Trop(X,\iota) \coloneqq \Trop(X \times_K L, \iota_L) $.

A natural question is, whether there is a universal tropicalization of a variety that is independent of the embedding. One way to construct this, is taking the projective limit of all tropicalizations with respect to all possible embeddings for a quasi-projective variety $X$. In the influential paper \cite{payne_analytification}, Payne showed that this limit space is homeomorphic to the \emph{Berkovich analytification} $X^{\an}$ of $X$. This analytic space $X^{\an}$, defined by Berkovich, is a connected Hausdorff topological space that contains the set of $K$-points $X(K)$ as a dense subset (see \cite{berkovich} for details). Payne's result led to ample research into limits of tropicalizations, for example in \cite{FGP_limits, KuronyaSouzaUlirsch, GiansiracusaGiansiracusa}. 

Recently, there have been two new developments that this paper combines:
\begin{enumerate}[(a)]
\item If one considers $X=\P^n$, in \cite{BKKUV} it was shown that the \emph{Goldman--Iwahori space} $\calXbar_n(K)$ is the limit of all tropicalizations with respect to \emph{linear} embeddings of $\P^n$ into higher-dimensional projective spaces. This space $\calXbar_n(K)$ consists of homothety classes of seminorms on $(K^{n+1})^{\ast}$ and has first been studied by Goldman and Iwahori as a piecewise linear analogue of symmetric spaces \cite{GoldmanIwahori} for the group $\PGL(K^{n+1})$. If $K$ is spherically complete, it coincides with (a compactification of) the affine Bruhat--Tits building of $\PGL(K^{n+1})$, \emph{e.g.}\ studied by Werner in \cite{Werner_seminorms}. The relation between affine buildings and Berkovich analytic spaces has been well established in \cite{berkovich, RemyThuillierWernerI, RemyThuillierWernerII, RemyThuillierWerner_survey}. 
\item In \cite{JSY22}, the authors define a ``real'' version of a Berkovich analytification: In this setup the field $K$ is real closed instead of algebraically closed and a \emph{real analytification} $X_r^{\an}$ of a variety $X$ is constructed, which takes into account the order on $K$. The authors introduce a non-Archimedean approach to real tropicalizations using the real analytification. Roughly speaking, given an affine variety $X$ over $K$, one restricts the tropicalization map to each orthant of $K^n$ and glues the resulting $2^n$ tropicalizations together. The resulting space $\Trop_r(X, \iota)$ is called the \emph{real tropicalization} of $X$ with respect to the embedding $\iota:X \xhookrightarrow{} K^n$. In \cite{JSY22}, it was shown that $X_r^{\an}$ is the limit of all real tropicalizations. Real tropicalizations have garnered particular attention due to Viro's patchworking \cite{viro_patchworking}, which was one of the earliest achievements of what is now called tropical geometry.
An equivalent perspective on real tropicalizations is given by real phase structures, and recently, in \cite{RRS23} the authors defined the real part of a smooth tropical variety equipped with a real phase structure, which is locally (the topological realization of) an oriented matroid \cite{RRS22}. 

\end{enumerate}

We merge these two approaches: Consider a real closed field $K$ equipped with a compatible non-Archimedean absolute value. An example is the field of real Puiseux series $\R\{\!\{t\}\!\} =\bigcup_{n \in \N} \R(\!(t^{\frac{1}{n}})\!)$, where the positive Puiseux series are those with positive leading coefficients. Let $X=\P^n$ be the projective space over $K$. We define the \emph{signed Goldman--Iwahori space} $\RX_n(K)$ to be the space of homothety classes of non-trivial signed seminorms on $(K^{n+1})^{\ast}$ (\emph{cf.} Definition \ref{def:signedGIspace}) equipped with the topology of pointwise convergence. We study the geometry of this space via real tropical geometry, in particular via real tropical linear spaces:

Let $I$ be the cofiltered category of linear embeddings $\P^n \hookrightarrow U \subseteq \P^m$, where $U$ is a torus-invariant open subset of $\P^m$ with morphisms given by commutative triangles 
\begin{center}
\begin{tikzcd}
\P^n \ar[hookrightarrow,r,"\iota"] \ar[hookrightarrow,dr,"\iota'"'] & U \ar[d] \\
 & U'
\end{tikzcd}
\end{center}
where $U\to U'$ is a composition of coordinate projections and permutations. For any linear embedding $\iota: \P^n \hookrightarrow U\subseteq \P^m$ as above, we construct a surjective and continuous map $\pi_{\iota}: \RX_n(K)\to \Trop_r(\P^n,\iota)$. For any morphism in $I$, the coordinate projections and permutations induce a continuous map between the respective real tropicalizations. This setup allows us to formulate and prove:

\begin{maintheorem}[Theorem \ref{thm:limits}]\label{mainthm:limits}
The canonical map 
$$\varprojlim_{\iota \in I} \pi_{\iota}\colon \RX_n(K) \longrightarrow\varprojlim_{\iota \in I} \Trop_r\big(\PP^n,\iota\big)$$ 
    is a homeomorphism.
\end{maintheorem}

Theorem \ref{mainthm:limits} can both be viewed as a signed version of \cite[Theorem A]{BKKUV} or as a linear version of \cite[Theorem 6.14]{JSY22}. Even though the signed Goldman--Iwahori space $\RX_n(K)$ is not a building in the sense of Bruhat and Tits, it can be seen as a signed analogue of the building of $\PGL$. From this perspective, it is a real analytic and non-Archimedean analogue of a symmetric space for $\PGL$.  It might be interesting to examine this space more closely in order to study reductive groups over real closed fields. Taking absolute values yields a natural, surjective, and continuous map $\RX_n(K)\to \calXbar_n(K)$ to the Goldman--Iwahori space (\emph{cf.} Section \ref{section:relGIspace}).

Valuated matroids are the combinatorial objects that naturally arise with (ordinary) tropicalizations of linear spaces. To every linear embedding $\iota: \P^n \hookrightarrow \P^m$ we may associate a valuated matroid that uniquely determines the tropicalization $\Trop(\P^n,\iota)$. Even in the non-realizable case, one can associate a linear space to a valuated matroid, which agrees with the Bergman fan in the case of trivial valuation. In \cite{BKKUV}, the construction was generalized to allow us to describe finite rank matroids on infinite ground sets. In \cite[Theorem C]{BKKUV}, it was also shown, that the Goldman--Iwahori space $\calXbar_n(K)$ is the tropical linear space associated to the \emph{universal realizable valuated matroid} $M_{\univ}$. This matroid has as ground set $(K^{n+1})^{\ast}$ and the rank function is given by the dimension of the subspace spanned by a subset. 

When the field $K$ is real closed (or, more generally, ordered), the additional decoration on a matroid is that of an orientation, which has been studied extensively \cite{orientedMatroidsBook, gebert_ziegler}. As matroids capture the combinatorics of linear dependence over fields, oriented matroids capture the combinatorics of linear dependence over ordered fields, taking into account signs of linear dependencies. If the field $K$ is in addition equipped with a compatible non-Archimedean value, we obtain an \emph{oriented valuated matroid}, which is a matroid with both decorations that satisfy a compatibility condition. 

Baker and Bowler introduce matroids with decorations in the language of hyperfields \cite{bakerbowler2016matroids}, where valuated/oriented/oriented valuated matroids arise as matroids over the tropical hyperfield $\T$/the sign hyperfield $\mathbb{S}$/the real tropical hyperfield $\R\T$.

To every oriented valuated matroid $\mathcal{M}$ on a ground set $E$, one can associate a \emph{real tropical linear space} $\Trop_r(\mathcal{M})$ \cite{tabera2015real, jurgens2018real}. If $\mathcal{M}$ is represented by a linear embedding over a real closed field, then the real tropicalization of the linear embedding coincides with the real tropical linear space associated to $\mathcal{M}$. This extends the trivial valuation case considered in \cite{ArdilaKlivansWilliams, celaya-diss} who shows that the real tropical linear space equals the real Bergman fan of the associated oriented matroid. 

We generalize these ideas and construct a real tropical linear space for any oriented valuated matroid of finite rank, whose ground set explicitly need not be finite (see Section \ref{section:Muniv}). We then show:
\begin{maintheorem}[Theorem \ref{thm:GIspaceIsTropMuniv}]\label{mainthm:GIspaceIsTropMuniv}
There is a homeomorphism
$$\RX_{n}(K) = \Trop_r(\Muniv)\, .$$
\end{maintheorem}
Here, $\Muniv$ denotes the \emph{universal realizable oriented valuated matroid} with ground set $E=(K^{n+1})^{\ast}$.
This gives $\RX_n(K)$ an appealing interpretation as the real tropicalization of the natural universal embedding 
\[\iota_{\univ}: \P^{n}\longhookrightarrow \P\left( K^{E}\right) \, .\] 
In particular, this allows for a notion of a real tropicalization of an embedding of a finite-dimensional variety into an infinite-dimensional vector space. 
This gives an interpretation of the signed Goldman--Iwahori space as a universal realizable real tropical linear space.

In the special case $K=\R$ with trivial valuation signed seminorms are always diagonalizable, which need not even be true for other trivially valued fields (\emph{cf.}\ Example \ref{ex:not_diagonalizable}). Therefore, in general, computing signed Goldman--Iwahori spaces turns out to be difficult. However, if $K=\R$, we can explicitly parameterize the signed Goldman--Iwahori space and determine the fibers of the map $\RX_n(\R)\to \calXbar_n(\R)$.

In the case of trivial valuation, the real tropical linear space of an oriented matroid agrees with its \emph{real Bergman fan} studied in \cite{celaya-diss}. We show that for the universal realizable oriented matroid, this is still true if $K=\R$, but fails for other real closed fields (\emph{cf.} Theorem \ref{thm:realBMfanBuilding} and Remark \ref{remark:infinitebergmanfan} (d)). Tropical linear spaces relate to tropical convexity (see \cite[\S10]{joswig2021essentials} for details), in the same way that real tropical linear spaces relate to signed tropical convexity \cite{loho2022signedtropicalhalfspacesandconvexity}. The real tropicalization of a linear space is the TC-convex hull of all $\R\T$-cocircuits of the associated $\R\T$-matroid \cite[Theorem 7.8]{loho2022signedtropicalhalfspacesandconvexity}. We give an interpretation of this result for the real tropicalization of the universal realizable oriented matroid $\calM_{\univ}$.

\subsection*{Conventions} We write $\overline{\R} = \R \cup \{\infty\}$. For any set $E$, we write $\T\P^E \coloneqq  \big(\overline{\R}^E\setminus \infty\big)/\R\mathbf{1}$, where $\mathbf{1}$ refers to the function with constant value $1$. If $E=\{0,\dots,n\}$ for some natural number $n$, we write $\T\P^n$ for $\T\P^E$.
For a field $K$, which will be clear from the context, we write $\A^n = \Spec(K[t_1,\dots,t_n])$ for affine space and $\P^n = \Proj( K[t_0,\dots,t_n])$ for projective space. If $K=\R$, we write $\R\P^n$ for the set of $\R$-points of $\P^n$ and consider this space with the Euclidean topology.
We denote the field with two elements by $\Z_2 = \{0,1\}$.
We consider the ordering $+,- > 0$ of signs which we extend componentwise to a partial order on sign vectors. Given an ordered field $K$, we denote by $\sgn: K \to \{0,+1,-1\}$ the sign function.

\subsection*{Acknowledgements}
This project was initiated 2023 at the  Graduate Student Meeting on Applied Algebra and Combinatorics in Stockholm, where we had the opportunity to present our project on buildings and tropical linear spaces \cite{BKKUV}. Kris Shaw pointed us to the real analytification and real tropicalization of Jell, Scheiderer, and Yu, thus inspiring us to investigate a real/signed version of our story.  We thank Kris Shaw and further all speakers and organizers at this occasion. Additionally, the authors want to thank Luca Battistella, Andreas Gross, Alex K\"uronya, Georg Loho, Oliver Lorscheid, Alejandro Martinez Mendez, Beatrice Pozzetti, Kemal Rose, Ben Smith, Pedro Souza, Martin Ulirsch, Alejandro Vargas, and Annette Werner for helpful conversations and discussions at various stages of the project.

\subsection*{Funding} This project has received funding from the Deutsche Forschungsgemeinschaft (DFG, German Research Foundation) TRR 326 \emph{Geometry and Arithmetic of Uniformized Structures}, project number 444845124, from the DFG Sachbeihilfe \emph{From Riemann surfaces to tropical curves (and back again)}, project number 456557832, as well as the DFG Sachbeihilfe \emph{Rethinking tropical linear algebra: Buildings, bimatroids, and applications}, project number 539867663, within the
SPP 2458 \emph{Combinatorial Synergies}. 


\section{Real Tropical Geometry} \label{section:realtropicalgeometry}

We describe the \emph{real part} of tropical projective space which is homeomorphic to the real projective space $\R\P^n$. This space is the ambient space of the real tropicalizations of subvarieties of $\P_K^n$ over a real closed field $K$ with compatible absolute value. We describe this construction of a real tropicalization, which, in contrast to ordinary tropicalization, also takes the order on $K$ into account. A more general and thorough treatment can be found for example in \cite{RRS23}.

\subsection{Tropical Patchworking}
Let $\overline{\R} = \R \cup \{\infty\}$. For $n\in \N$, the \emph{tropical projective space} of dimension $n$ is given by $\T\P^n\coloneqq  \overline{\R}^{n+1} \setminus \{(\infty, \dots, \infty )\} /\R \mathbf{1}$, where $\mathbf{1}=(1,\dots,1)$. Via the identification $\overline{\R}\cong \R_{\geq 0},\ x \mapsto \exp(-x)$, we can identify $\T\P^n \cong \R_{\geq 0}^{n+1} \setminus \{(0,\dots,0)\} / ( \R_{>0})$. The space $\T\P^n$ is naturally stratified, where the closures of strata are given by $\{(x_0:\dots:x_n)\in \T\P^n\mid x_i=\infty\  \forall i \in I\}$, where $I\subsetneq \{0,\dots,n\}$.

We can glue together $2^n$ symmetric copies of $\T\P^n$ indexed by $\varepsilon\in \Z_2^{n+1}/\mathbf{1}$ to obtain a space
\[\RTP^n \coloneqq  \bigcup_{\varepsilon \in \Z_2^{n+1}/\mathbf{1}} \T\P^n(\varepsilon) /\sim. \]
Here, $\sim$ identifies the codimension-one strata $\{(x_0:\dots: x_n)\in \T\P^n(\varepsilon) \mid x_i=\infty\}$ and $\{(x_0:\dots: x_n)\in \T\P^n(\varepsilon') \mid x_i=\infty\}$ if $\varepsilon+\varepsilon'=e_i$ in $\Z_2^{n+1}/\mathbf{1}$. The identification $\T\P^n \cong \R_{\geq 0}^{n+1} \setminus \{(0,\dots,0)\} / ( \R_{>0})$ yields an identification
\begin{align*}
\RTP^n &\cong \R\P^n, \\
(\varepsilon, (x_0:\dots:x_n)) &\mapsto \bigl((-1)^{\varepsilon_0}\exp(-x_0):\dots:(-1)^{\varepsilon_n}\exp(-x_n)\bigr).
\end{align*}

There is an obvious continuous retraction map $\lvert \cdot \rvert_{\R\P^n}: \R\P^n \to \R_{\geq 0}^{n+1} \setminus \{(0,\dots,0)\} / ( \R_{>0}) \cong \T\P^n$, sending each $(x_0:\dots : x_n) \in \R\P^n$ to $(|x_0|:\dots:|x_n|)$, where $\lvert \cdot \rvert$ denotes the absolute value on the real numbers. 

\begin{remark}
This is a special case of the construction $\mathfrak{R}\T \Sigma$ in \cite{RRS23, gelfandKapranovZelevinsky} associated to a pointed polyhedral fan $\Sigma$, which is a space homeomorphic to the real part of a complex toric variety (in our setting $\P^n_{\C}$) by gluing together multiple symmetric copies of the tropical toric variety (in our setting $\T\P^n$).  In \cite{RRS23} this is the ambient space of the \emph{patchworks} of tropical varieties.
\end{remark}

\begin{example}
We make the space $\RTP^2 \cong \R\P^2$ explicit by gluing four triangles, each homeomorphic to $\T\P^2$, together as in Figure \ref{fig:RTP2}
and identifying antipodal points. The retraction map $\lvert \cdot \rvert_{\RP^2}$ folds up the four triangles to one.
\begin{center}

\begin{figure}[h]
\begin{tikzpicture}[scale=.5]

\coordinate (0) at (0,0);
\coordinate (1) at (3,0);
\coordinate (2) at (-3,0);
\coordinate (3) at (0,3);
\coordinate (4) at (0,-3);

\foreach \i in {0,1,2,3,4} {\path (\i) node[circle, black, fill, inner sep=1]{};}
\draw[black] (1)--(2)--(3)--(4)--(1)--(3)--(2)--(4);

\draw[->] (6,0)--(7,0);
\coordinate (5) at (9,-1);
\coordinate (6) at (12,-1);
\coordinate (7) at (9,2);
\draw[black] (5)--(6)--(7)--(5) ;

\end{tikzpicture}
\caption{The map $\RTP^2\to \T\P^2$.}
\label{fig:RTP2}
\end{figure}
\end{center}
\end{example}
For the sake of clarity, we will henceforth always use the coordinates of $\R\P^n$ and omit any logarithms. The downside is, that our tropicalizations are not polyhedral complexes in these coordinates. However, the formulas are much clearer since these coordinates make it easier to keep track of the signs.

\subsection{Ordered and Real Closed Fields}
We recall the definition of an \emph{ordering} on a ring: For a commutative ring $A$, a subset $P \subseteq A$ is called an \emph{ordering} of $A$ if $P+P \subseteq P, P \cdot P \subseteq P, P \cup -P = A,$  and $P \cap -P $ is a prime ideal of $A$. The support of $P$ is $\text{supp}(P) \coloneqq  P \cap -P$. To $P$ we associate an order relation and a sign function in the natural way: Write $f >_P 0$ and $\sgn_P(f) = +1$  if $-f \notin P$, write $f \geq_P 0$ if $f \in P$, and $\sgn_P(f) = 0$ if $f \in \text{supp}(P)$. We have $\sgn_P(-f) = - \sgn_P(f)$. We may refer to the ordering as either $P$ or $<_P$.  An ordered field $K$ is \emph{real closed} if every non-negative element $x \geq 0$ has a square root in $K$. If $K$ is real closed, there is a unique ordering on $K$ defined by taking the positive elements to be precisely the non-zero squares. Every ordered field has a real closed extension, called the \emph{real closure}. Whenever a real closed field is equipped with a non-Archimedean absolute value $\lvert \cdot \rvert_K$, we assume that the absolute value is \emph{compatible} with the (unique) ordering. This means, if $0 \leq a \leq b$ then $ |a|_K \leq |b|_K$, or equivalently, if $|a|_K > |b|_K $ then $ \sgn(a+b) = \sgn(a) = \sgn(a-b)$.

\begin{example}\label{ex:puiseux_series}
Let $\R\{\!\{t\}\!\}$ be the field of \emph{real Puiseux series}, \emph{i.e.}, 
\[ \R\{\!\{t\}\!\} =\bigcup_{n \in \N} \R(\!(t^{\frac{1}{n}})\!)\, .\] Every element of  $\R\{\!\{t\}\!\}$ is a power series with real coefficients and rational exponents
with bounded denominator. This field is equipped with the absolute value that maps $f=\sum_{q \in \Q} a_q t^q$ to $\exp(-q_0)$, where $q_0$ is the minimal index such that $a_{q_0}\neq 0$. Moreover, let $\sgn (f)\coloneqq \sgn(a_{q_0})$, \emph{i.e.}, the positive Puiseux
series are those with positive leading coefficients. Then $\R\{\!\{t\}\!\}$ is a real closed field whose absolute value is compatible with the order.
\end{example}

\subsection{Real Tropicalization} \label{section:realtropicalization}

We will begin by providing a brief overview of ordinary tropicalization. Let $K$ be an algebraically closed field with a non-trivial non-Archimedean absolute value $\lvert \cdot \rvert_K: K\to \R_{\geq 0}$. Denote the \emph{tropicalization map} by 
\begin{align*}
\trop: \P^n(K) &\longrightarrow (\R_{\geq 0}^{n+1} \setminus \{0\}) / \R_{>0}\, ,\\
(x_0: \dots :x_n) &\longmapsto (|x_0|_K : \dots : |x_n|_K).
\end{align*}
For any closed subvariety $\iota: X\hookrightarrow \P^n  $ the \emph{tropicalization} $\Trop(X,\iota)$ is defined to be the closure of the image of $X$ under $\trop$. If $K$ is trivially valued or not algebraically closed, we pass to a non-trivially valued, algebraically closed extension $L/K$ and define $\Trop(X,\iota)\coloneqq \Trop(X_L,\iota_L) \subseteq (\R_{\geq 0}^{n+1} \setminus \{0\}) / \R_{>0}$, where $X_L= X\times_K L$. This is independent of the choice of $L$ (\emph{e.g.}\ by \cite[Proposition 3.8]{Gubler_guide}).

\begin{remark}
It is more standard to write $\trop$ in the coordinates of $\T\P^n$, for which we have $$\trop(x_0:\dots:x_n)=(-\log |x_0|_K : \dots : -\log |x_n|_K)\, .$$
In these coordinates, the Bieri--Groves theorem \cite[Theorem A]{bierigroves1984geometry} and \cite[Theorem 2.2.3]{einsiedlerkapranovlind2006} states that, if $X$ is irreducible, then $\Trop(X, \iota)$ has the structure of a pure-dimensional rational polyhedral set of the same dimension as $X$.
\end{remark}

Let now $K$ be real closed instead of algebraically closed. We assume that $\lvert \cdot \rvert_K$ is compatible with the unique ordering on $K$. 
The \emph{real tropicalization map} is given by 
\begin{align*}
\trop_r: \P^n(K) &\longrightarrow \R\P^n\, , \\
(x_0: \dots :x_n) &\longmapsto (\sgn(x_0) |x_0|_K : \dots :\sgn(x_n) |x_n|_K).
\end{align*}
Note that this map takes into account the order on $K$. This can be seen as an orthantwise tropicalization, where we restrict the tropicalization map to each orthant and glue the resulting tropicalizations together. The affine version of this map is considered by Jell, Scheiderer, and Yu in \cite{JSY22} and \cite{tropicalspectahedra}, where the authors study properties of the images of semialgebraic sets under the real tropicalization map. The logarithmic version of this construction, without signs, was used by Alessandrini in \cite{alessandrini2013logarithmic} who showed that the logarithmic limit of a real semialgebraic set is a polyhedral complex.

As before, for every closed subvariety (or, more generally, for every semialgebraic subset) $\iota: X\hookrightarrow \P^n$, we define the \emph{real tropicalization} $\Trop_r(X, \iota) \subseteq \R\P^n$ as the closure of the image of $X(K)$ under $\trop_r$. If $K$ is trivially valued, we choose a non-trivially valued, real closed field extension $L/K$ such that the absolute value on $L$ extends the one
on $K$ and define $\Trop_r(X, \iota)\coloneqq \Trop_r(X_L, \iota_L)\subseteq \R\P^n$. The real tropicalization is independent of the choice of extension by \cite[Theorem 6.9]{JSY22}. By construction, the map $\lvert \cdot \rvert_{\R\P^n}: \R\P^n \to \T\P^n$ maps $\Trop_r(X,\iota)$ to $\Trop(X,\iota)$.

\begin{example}
Let $K=\R$ be equipped with the trivial valuation and let  $X=V(x_0+x_1-x_2)\subseteq \P^2_{\R}$. We want to compute $\Trop_r(X,\iota)$, where $\iota: X\hookrightarrow \P^2$ is the inclusion. Consider the base change to the field $L$ of real Puiseux series. Let $(y_0:y_1:y_2) \in X_L(L)=V(x_0+x_1-x_2)\subseteq \P^2_L$, \emph{i.e.}, $y_0+y_1=y_2$. Then 
$$\sgn((y_0,y_1,y_2))\in \{(+,+,+),(-,-,-),(+,-,-),(-,+,+),(+,-,+),(-,+,-)\},$$
\emph{i.e.}, if $\sgn(y_0)=\sgn(y_1)$, then all signs are the same. 
By the axioms of a non-Archimedean valuation it is clear that $|y_2|\leq \max(|y_0|,|y_1|)$ and that, if $|y_0|\neq |y_1|$, we have equality. Moreover, we also have equality if $\sgn(y_0)=\sgn(y_1)$. Finally, if $\sgn(y_0)\neq \sgn(y_1)$ and $|y_0|= |y_1|$, then $\sgn(y_2)\cdot|y_2|$ can have any value in $[-|y_0|,|y_0|]\cap \Q$. We may choose representatives in $\P^2$ such that $\max\{|y_0|,|y_1|,|y_2|\}=1$ and $y_2\geq 0$. In Figure \ref{fig:tropLine} the real tropicalization $\Trop_r(X,\iota)$ is now indicated in red and the (usual) tropicalization $\Trop(X,\iota)$ in blue.

\begin{figure}[h]

\begin{center}
\begin{tikzpicture}[scale=.5]

\coordinate (0) at (0,0);
\coordinate (1) at (3,0);
\coordinate (2) at (-3,0);
\coordinate (3) at (0,3);
\coordinate (4) at (0,-3);

\coordinate (8) at (-1.5,1.5);
\coordinate (9) at (-1,1);
\coordinate (10) at (1,1);
\coordinate (11) at (1,-1);
\coordinate (12) at (1.5,-1.5);

\foreach \i in {0,1,2,3,4} {\path (\i) node[circle, black, fill, inner sep=1]{};}
\draw[black] (1)--(2)--(3)--(4)--(1)--(3)--(2)--(4);
\draw[red] (8)--(9)--(10)--(11)--(12);

\node[below left=0mm of 0] {\small (0,0,1)};
\node[right] at (1) {\small (1,0,0)};
\node[left] at (2) {\small (-1,0,0)};
\node[above] at (3) {\small (0,1,0)};
\node[below] at (4) {\small (0,-1,0)};

\draw[->] (6,0)--(7,0);

\coordinate (5) at (9,-1);
\coordinate (6) at (12,-1);
\coordinate (7) at (9,2);
\draw[black] (5)--(6)--(7)--(5) ;
\draw[GoetheBlue](9,0)--(10,0)--(10.5,0.5)--(10,0)--(10,-1);

\end{tikzpicture}
\caption{The map $\RTP^2\to \T\P^2$.}
\label{fig:tropLine}
\end{center}
\end{figure}

\end{example}

\begin{remark}\label{rem:realPhaseStructures}
In \cite{RRS23} the authors introduce \emph{real phase structures} on rational polyhedral spaces and tropical varieties.
For a rational polyhedral subspace $X$ in $\R^n$, a real phase structure is an assignment of an affine subspace of $\Z_2^n$ to each facet of $X$. The assignment of a real phase structure is used to describe the \emph{patchwork} of a tropical variety, which should be thought of as its real part. The authors show that a patchwork describes, up to homeomorphism, fibers of real analytic families with non-singular tropical limits. As explained in \cite[\S 4.8]{RRS23} their construction of a real tropicalization, \emph{i.e.}, the patchwork, agrees with the logarithmic version of the real tropicalization from \cite{JSY22} that we described above and use in this paper.
\end{remark}

\section{Matroids over the Real Tropical Hyperfield}\label{section:MatroidsHyperfield}

Matroids over hyperfields due to Baker and Bowler \cite{bakerbowler2016matroids, bakerbowlerpartial} simultaneously generalize linear subspaces, matroids, oriented matroids, and valuated matroids. We recall their notion as well as the hyperfields that are of our interest such as the sign hyperfield, the tropical hyperfield, and, most importantly, the real tropical hyperfield. Matroids over the sign hyperfield correspond to \emph{oriented matroids}, while matroids over the tropical hyperfield correspond to \emph{valuated matroids}. \emph{Oriented valuated matroids} are hybrid objects combining oriented matroids and valuated matroids in a compatible way. From the perspective of matroids over hyperfields they arise naturally as matroids over the real tropical hyperfield $\R\T$. We recall the definition of a real tropical linear space associated to an oriented valuated matroid defined via its $\R\T$-circuits.

\subsection{Hyperfields}
In \cite{viro2010hyperfields} Viro introduced hyperfields as a convenient technique in tropical geometry. Especially in the last few years there has been a surge of interest and research in the realm of hyperfields and tropical geometry \cite{lorscheid_hyperfield, bakerbowler2016matroids, baker_lorscheid_moduli, maxwell2023geometry, maxwell2024generalising}.
A \emph{hyperfield} $\mathbb{H}$ is a set with a multiplication $\cdot$ and addition $\oplus$, where addition may be multivalued, that satisfies axioms similar to those for a field. Several of the following hyperfields were first introduced in Viro's paper \cite{viro2010hyperfields}, to which we refer for precise definitions. 

\begin{example}\label{exa:hyperfields}
\ 
\begin{enumerate}[(a)]
\item Any field $K$ can trivially be considered a hyperfield with its ordinary multiplication and addition, where we consider the sum of two elements as a singleton set.
\item The \emph{Krasner hyperfield} $\mathbb{K}$ as a set is $\mathbb{K} = \{0,1\}$ with the usual multiplication and $-1 = 1$. For addition $0$ is the neutral element and $1 \oplus 1 = \mathbb{K}$. 
\item The \emph{sign hyperfield} $\mathbb{S}$ on elements $\{0,+1,-1\}$ has multiplicative group $(\{\pm 1\},\cdot)$. The addition is given by $0 \oplus x = 0$ for all $x \in \mathbb{S}$, $1 \oplus 1 = 1, -1 \oplus -1 = -1$, and $-1 \oplus 1 = \mathbb{S}$. 
\item The \emph{tropical hyperfield} $\TT$ on elements $\R_{\geq 0}$ with multiplicative notation has as multiplicative group $(\R_{>0},\cdot)$, and the addition is defined by
\begin{align*}
a \oplus b &=\begin{dcases}
\max(a,b) \quad \quad \text{     if } a \neq b,\\
[0,a] \quad  \quad \quad  \quad \text{     if } a = b.\\
\end{dcases}
\end{align*}  
\item The \emph{real tropical hyperfield} $\R\T$  on elements $\R$ has multiplicative group $(\R^*,\cdot)$ with addition
\begin{align*}
a \oplus b &=\begin{dcases}
a \quad \quad \text{     if } |a| > |b|,\\
b \quad  \quad \text{     if } |a| < |b|,\\
a \quad  \quad \text{     if } a = b,\\
[a,b] \quad \text{if } a = -b \leq 0,\\
[b,a] \quad \text{if } a = -b \geq 0.\\   
\end{dcases}
\end{align*}    
\end{enumerate}
\end{example}

A  \emph{homomorphism of hyperfields} is a map $f: \mathbb{H}_1 \longrightarrow \mathbb{H}_2$ such that $f(0) = 0, f(1) = 1, f(x \cdot y) = f(x) \cdot f(y)$, and $f(x \oplus y) \subseteq f(x) \oplus f(y)$ for $x,y \in \mathbb{H}_1$.

\begin{example} \label{exa:hyperfieldhom}
\ 
\begin{enumerate}[(a)]
\item Every hyperfield $\HH$ has a canonical homomorphism into $\mathbb{K}$ by sending $0$ to $0$ and every non-zero element to $1$.
\item The maps $\R\T \longrightarrow \T, a \mapsto |a|$ taking the absolute value and $\R\T \longrightarrow \mathbb{S}, a \mapsto \sgn(a)$ keeping only the sign information are hyperfield homomorphisms.
\item For a field $K$, a map $K\to \mathbb{S}$ is a hyperfield homomorphism if and only if it is of the form $\sgn_P$ for an ordering $P\subset K$.
\item For a field $K$, a map $K\to \T=\R_{\geq 0}$ is a hyperfield homomorphism if and only if it is a non-Archimedean absolute value. 
\item Let $K$ be an ordered field with a compatible absolute value $\lvert \cdot \rvert_K$. The natural signed absolute value map $K \longrightarrow \R\T, x \mapsto \sgn(x) \cdot |x|_K$ is a hyperfield homomorphism.
\end{enumerate} 
\end{example}

\subsection{Matroids over Hyperfields}

\begin{definition}[{\cite{bakerbowler2016matroids}}] \label{def:matroidsoverhyperfieldsGP}
Let $E$ be a non-empty finite set, $\HH$ a hyperfield, and let $n$ be a positive integer. A \emph{Grassmann--Plücker function of rank $n$ on $E$ with coefficients in $\HH$} is a function $\varphi: E^n \longrightarrow \HH$ such that:
\begin{enumerate}[(a)]
\item $\varphi$ is not identically zero,
\item $\varphi$ is \emph{alternating}, \emph{i.e.},
\begin{align*}
\varphi(x_1,\dots,x_i,\dots,x_j,\dots,x_n) &= -\varphi(x_1,\dots,x_j,\dots,x_i,\dots,x_n) \text{ and}\\
\varphi(x_1,\dots,x_n) &= 0 \quad \text{if } x_i=x_j \text{ for some }i,j\, .
\end{align*}
\item (Grassmann-Plücker relations) For any two subsets $\{x_1,\ldots,x_{n+1}\}, \{y_1,\ldots,y_{n-1}\} \subseteq E$, 
\[
0 \in \bigoplus_{k=1}^{n+1} \varphi(x_1,\ldots,\hat{x}_k,\ldots,x_{n+1}) \cdot \varphi(x_k,y_1,\ldots,y_{n-1})\, .
\]
\end{enumerate}
Two Grassmann--Plücker functions $\varphi_1$ and $\varphi_2$ are \emph{equivalent} if $\varphi_1 = \alpha \cdot \varphi_2$ for some $\alpha \in \HH^{\times}$.
A  \emph{(strong) matroid over a hyperfield} $\HH$  on $E$ of rank $r$ is an equivalence class of a Grassmann--Plücker function of rank $r$ on $E$.
\end{definition}

\begin{remark}
Baker and Bowler define strong and weak matroids over hyperfields, which coincide for doubly distributive hyperfields by \cite[Theorem 5.4]{bakerbowlerpartial}. Since all the hyperfields we consider (see Example \ref{exa:hyperfields}) are  doubly distributive hyperfields, we will simply refer to matroids over hyperfields.
\end{remark}

For the hyperfields of Example \ref{exa:hyperfields}, we obtain the following (classical) notions:
\begin{enumerate}[(a)]
\item A matroid over a field $K$ is the same thing as a linear subspace of $K^E$ of rank $n$. This is the classical representation of a subspace via its Plücker coordinates that satisfy the Grassmann--Plücker relations.
\item When $\HH = \mathbb{K}$ is the Krasner hyperfield, a matroid over $\mathbb{K}$ is the same as a usual matroid since the Grassmann--Plücker relations are equivalent to the basis exchange axiom for matroids.
\item A matroid over the sign hyperfield $\mathbb{S}$ is the same as an \emph{oriented} matroid, where a choice of a Grassmann--Plücker $\varphi$ function is called a \emph{chirotope} (note that $\varphi,-\varphi$ are the only chirotopes). This choice gives us then a notion of an \emph{oriented basis}, where $(b_1,\dots,b_n)$ is a \emph{positively oriented basis}, if $\varphi(b_1,\dots,b_n)=1$.
\item A matroid over the tropical hyperfield $\T$ is the same as a \emph{valuated} matroid in the sense of Dress--Wenzel \cite{dress1992valuated}. The Grassmann--Plücker function is usually called a tropical Plücker vector.
\item A matroid over the real tropical hyperfield $\R\T$ is an \emph{oriented valuated} matroid. This particular case holds significant interest within this paper.
\end{enumerate}

Matroids over hyperfields admit a useful pushforward operation (\emph{cf.}\ \cite[\S 4.2]{bakerbowler2016matroids}): Given a Grassmann–Plücker function $\varphi: E^n \xrightarrow{} \HH_1$ and a homomorphism of hyperfields $f: \HH_1 \xrightarrow{} \HH_2,$ the pushforward $f_* \varphi: E^n \xrightarrow{} \HH_2$ is defined by the formula
\begin{align*}
f_*\varphi(e_1,\ldots,e_n) = f(\varphi(e_1,\ldots,e_n)).    
\end{align*}
This is again a Grassmann--Plücker function. If $\mathcal{M}$ is a matroid over $\HH_1$ given by $\varphi$, the $\emph{pushforward}$ $f_*(\mathcal{M})$ is defined by $f_*\varphi$ and is a matroid over $\HH_2$. 

\begin{definition}
Let $f: \HH_1 \xrightarrow{} \HH_2$ be a homomorphism of hyperfields, and let $\mathcal{M}_2$ be a matroid on $E$ with coefficients in $\HH_2$. We say that $\mathcal{M}_2$ is \emph{realizable} with respect to $f$ if there is a matroid $\mathcal{M}_1$ over $\HH_1$ such that $f_*(\mathcal{M}_1) = \mathcal{M}_2.$
\end{definition}

\begin{example}
\label{exa:pushforwardsofmatroids}
\ 
\begin{enumerate}[(a)]
\item Let $\HH$ be a hyperfield and let $\omega: \HH \longrightarrow \mathbb{K}$ be the canonical hyperfield homomorphism. If $\mathcal{M}$ is an $\HH$-matroid, the pushforward $\omega_*(\mathcal{M})$ is called the \emph{underlying matroid}.
\item Let $K$ be an ordered field and let $\sgn \colon K \xrightarrow{} \mathbb{S}$ be the sign map. If $W \subseteq K^m$ is a linear subspace (considered in the natural way as a $K$-matroid), the pushforward $\sgn_*(W)$ coincides with the oriented matroid which one traditionally associates to $W$.
\item  Let $K$ be a field with a non-Archimedean absolute value $\lvert\cdot\rvert_K: K \longrightarrow \T$ (considered as a hyperfield homomorphism) and let $W \subseteq K^n$ be a linear subspace. Then the pushforward $\lvert \cdot \rvert_*(W)$ is the valuated matroid associated to $W$. 
\item Let $K$ be a real closed field with a compatible absolute value $\lvert \cdot \rvert_K$. If $W \subseteq K^m$ is a linear subspace, the pushforward of $W$ under the signed absolute value map $K \longrightarrow \R\T, x \mapsto \sgn(x) \cdot |x|_K$ is the oriented valuated matroid associated to $W$. Explicitly, let $W$ be the row-space of the $n \times m$-matrix $[v_1,\dots,v_m].$ Let $E = \{v_1,\ldots,v_m\}$ be the set of column vectors of the matrix which form a spanning set of vectors of $K^n$. The associated Grassmann--Plücker function of the $\R\T$-matroid associated to $W$ is given by
\begin{equation*}\begin{split}
E^n &\longrightarrow \R, \\
(v_{a_1},\ldots,v_{a_n}) &\longmapsto \sgn(\det[v_{a_1},\dots,v_{a_n}]) \cdot \lvert\det[v_{a_1},\dots,v_{a_n}]\rvert_K  \ .
\end{split}\end{equation*}

\item Consider an $\R\T$-matroid (\emph{i.e.}, an oriented valuated matroid) $\mathcal{M}$. The pushforwards of $\mathcal{M}$ under the natural maps $\R\T \longrightarrow \T$ and $\R\T \longrightarrow \mathbb{S}$ yield the underlying valuated matroid and the underlying oriented matroid, respectively.
\end{enumerate}
\end{example}

\subsection{Oriented Valuated Matroids}
Since oriented valuated matroids are not as well studied as for example oriented or valuated matroids, we describe them now in more detail. The following explicit description is taken from \cite[\S 1.5]{giansiracusa2023e_}:

Let $\mathcal{M}$ be an oriented valuated matroid given by the Grassmann--Plücker function $\varphi: E^n \longrightarrow  \R.$ Condition $(3)$ in Definition \ref{def:matroidsoverhyperfieldsGP} explicitly means: For each $\{x_1,\ldots,x_{n+1}\} \in E^{n+1}$ and $\{y_1,\ldots,y_{n-1}\} \in E^{n-1}$, either the numbers
$$ \{(-1)^k \varphi(x_1,\ldots, \hat{x}_k,\ldots,x_{n+1}) \cdot \varphi(x_k,y_1,\ldots,y_{n-1}) \}_{k=1,\dots,n+1}$$ are all zero, or the maximum modulus occurs with both signs. The group $\R\T^* = \R^*$ acts on the set of such $\varphi$ by multiplication, and an oriented valuated matroid is an orbit.

As in Example \ref{exa:pushforwardsofmatroids} (e) the pushforwards of an $\R\T$-matroid $\mathcal{M}$ yield an underlying valuated matroid given by a tropical Plücker vector $v=|\varphi|: E^{n} \longrightarrow \T = \R_{\geq 0}$ and an underlying oriented matroid given by the chirotope (\emph{i.e.}, Grassmann--Plücker function) $\sgn \circ \varphi: E^n \longrightarrow \mathbb{S} = \{0,+1,-1\}$.
An equivalent definition of an oriented valuated matroid is also given by a \emph{compatible} pair of a tropical Plücker vector and a chirotope, see \cite[Definition 1.5.1]{giansiracusa2023e_}.

\begin{remark}
The more common, but equivalent, notion of a valuated matroid arises by taking $v = -\log|\varphi|$, where one takes hyperfield addition using $\min$ and multiplication given by usual addition.
\end{remark}

There is a cryptomorphic definition of $\HH$-matroids via $\HH$-circuits. For $\HH \in \{\mathbb{K}, \mathbb{S}, \mathbb{T}\}$ these give the usual axioms for circuits, signed circuits and valuated circuits. The general cryptomorphic axiomatization of $\HH$-circuits can be found in \cite[Definition 3.7]{bakerbowler2016matroids}. For $\HH = \R\T$, we will now state the axioms for the set of $\R\T$-circuits. For any $C,C' \in {\R\T}^E$ we define the \emph{support} of $C$ as $\Supp{C}:= \{e \in E \mid C_e \neq 0\}$, and we define the \emph{composition}
\begin{align*}
(C \circ C')_e = \begin{cases}
C_e \quad \text{if } |C_e| \geq |C'_e|,\\
C'_e \quad \text{otherwise}. \\
\end{cases}
\end{align*}

\begin{proposition} [{\cite[Theorem 45]{bowler2019perfectmatroidsoverhyperfields}}]\label{prop:circuit_description}
Let $E$ be a finite set. An $\R\T$-matroid on $E$ is equivalent to a subset $\mathcal{C} \subseteq {\R\T}^E$ that satisfies the following circuit axioms:
\begin{itemize}
\item[\textbf{(C0)}] $0 \notin \calC$,
\item[\textbf{(C1)}] if $C \in \calC$ and $\alpha \in {\R\T}^{*} = \R^*$, then $\alpha \cdot C \in \calC$,
\item[\textbf{(C2)}] if $C,C' \in \calC$ and $\Supp{C} \subseteq \Supp{C'}$, then there exists an $\alpha \in (\R\T)^* = \R^*$ such that $C' = \alpha \cdot C$,
\item[\textbf{(C3)}] for any $C,C' \in \calC, e,f \in E$ such that $C_e = -C'_e \neq 0$ and $|C_f| > |C'_f|$, there exists a $C'' \in \calC$ such that $C''_e = 0,\ C''_f = C_f$, and $|C''_g| < |C_g \circ C'_g|$ or $C''_g \in C_g \oplus C'_g$ for all $g \in E$.
\end{itemize}
\end{proposition}

We call $\R\T$-circuits \emph{signed valuated circuits}. 
\begin{lemma}[{\cite[Lemma 4.5] {bakerbowler2016matroids}}]\label{lem:pushforwardcircuits}
If $f: \HH_1 \xrightarrow{} \HH_2$ is a homomorphism of hyperfields and $\mathcal{M}$ is an $\HH_1$-matroid on $E$ with set of circuits $\mathcal{C}$, then the set of circuits of the pushforward $f_*(\mathcal{M})$ is given by 
$$ \{c f_*(C) \mid c \in \HH_2^{\times}, C \in \mathcal{C} \}\, .$$
\end{lemma}

In particular, the set of supports of $\mathcal{C}$ is the set of circuits of the underlying matroid.

\begin{example} \label{example:realizableorvalmatroid}
Let $K$ be an ordered field with a compatible absolute value $\lvert \cdot \rvert_K$. An $\R\T$-matroid $\mathcal{M}$ is realizable over $K$ if there is a $K$-matroid, \emph{i.e.},\ a linear space $W \subseteq K^E$, such that the pushforward of $W$ under the signed absolute value map equals $\mathcal{M}$. As in Example \ref{exa:pushforwardsofmatroids} (d) let $E = \{v_1,\ldots,v_m\}$ be a spanning set of vectors of $K^n$ and $W$ the row-space of the $n \times m$-matrix $[v_1,\dots,v_m]$. The set of circuits of $W$ is given by the vectors with minimal support in $E$. By Lemma \ref{lem:pushforwardcircuits} the set of signed valuated circuits, (\emph{i.e.}\ $\R\T$-circuits) of $\mathcal{M}$, is given by
\begin{align*}
\mathcal{C} &= \{\big (\sgn(\lambda_1)|\lambda_1|_K,\ldots,\sgn(\lambda_m)|\lambda_m|_K \big) \in \R^m \mid \\
&\quad \sum_i\lambda_i v_i= 0 \text{ is a minimal linear dependence} \}\, .
\end{align*}
This generalizes signed circuits of an oriented matroid by considering the trivial valuation on $K$.
\end{example}

\subsection{Real Tropical Linear Spaces} \label{section:realtropicallinearspaces}

Analogous to the definition of the tropical linear space of a valuated matroid \cite{speyer_tropical_linear_spaces}, the real tropical linear space $\Trop_r(\calM)$ associated to an oriented valuated matroid $\calM$ is defined as follows: 

\begin{definition}[{\cite[1.2.7]{jurgens2018real}}] \label{def:realtropicallinearspace}
Let $\mathcal{M}$ be an oriented valuated matroid on the ground set $E$. For a signed valuated circuit $C \in \calC$, we define the \emph{real tropical hyperplane}
\begin{align*}
\Trop_r(\mathcal{M}_C) &= \{(y_0:\ldots:y_m) \in {\R\P}^m \mid 0 \in \bigoplus_{e \in \Supp{C}} y_e \cdot C_e \} \\
&= \{(y_0:\ldots:y_m) \in {\R\P}^m \mid \text{ there exist indices } i \neq j  \text{ such that} \\
&\qquad \max_{e \in \Supp C}(|y_e|\cdot |C_e|) \text{ is attained at } i,j \text{ and } y_i \cdot C_i = - y_j \cdot C_j
\} \, .
\end{align*}
The \emph{real tropical linear space} associated to $\mathcal{M}$ is defined as
\begin{align*}
\Trop_r(\mathcal{M}) \coloneqq \bigcap_{C \in \calC} \Trop_r(\mathcal{M}_C).
\end{align*}
\end{definition}

\begin{remark}
This definition of the real tropical linear space is exactly the zero set of linear polynomials over the real tropical hyperfield $\R\T$.
\end{remark}

\begin{remark}
If $K$ is trivially valued, $\Trop_r(\mathcal{M})$ equals the \emph{real Bergman fan} of \cite{celaya-diss}, see Section \ref{section:realbergmanfans}. In the general case, the real tropical linear space of a matroid over the real tropical hyperfield is the so-called \emph{TC-convex hull} of the $\R\T$-circuits of $\mathcal{M}$ \cite{loho2022signedtropicalhalfspacesandconvexity}. The relation to signed tropical convexity is elaborated further in Remark \ref{rem:TropicalConvexity}.
\end{remark} 

The tropicalization of a linear space over a non-Archimedean field is determined by its associated valuated matroid {\cite[Proposition 4.2]{speyer_tropical_linear_spaces}}. If $K$ is a real closed field with a compatible absolute value $\lvert \cdot \rvert_K$, then similarly the real tropicalization of a linear space over $K$ is determined by the associated oriented valuated matroid.

\begin{proposition}[{\cite[Theorem 3.14]{tabera2015real}}, {\cite[Theorem 1.2.11]{jurgens2018real}}] \label{prop:tableratropicalbasis}
Let $K$ be a real closed field with a compatible absolute value $\lvert \cdot \rvert_K$. Let $\iota=(f_0:\dots:f_m): \P^n \xhookrightarrow{} \P^m$ be a linear embedding, and $\calM_{\iota}$ be the associated realizable oriented valuated matroid on $\{f_0,\dots,f_m\}\subset (K^{n+1})^*$ (\emph{cf.}\ Example \ref{example:realizableorvalmatroid}). Then
\[\Trop_r(\P^n,\iota)=\Trop_r(\calM_{\iota})\, .\]
\end{proposition}

In other words, the real tropicalization of a linear space only depends on the associated oriented valuated matroid.

\section{Real Analytification and Tropicalization} \label{section:realanalytificationandtrop}

Throughout this section, we fix a real closed, non-Archimedean valued field $K$ with absolute value $\lvert \cdot \rvert_K$ that is compatible with the order on $K$. We recall the real analytification of a variety over $K$ which was introduced in \cite{JSY22}. We then show that, similarly to the  Berkovich analytification, the real analytification of $\P^n$ can be described as a quotient of $\A^{n+1,\an}_r\setminus \{0\}$, and hence can be described in terms of signed multiplicative seminorms on a polynomial ring. This gives a real tropicalization map from $\P_r^{n,\an}$ to $\R\P^n$ extending the real tropicalization map from affine varieties to closed subvarieties of $\P^n$.

\subsection{The Real Analytification \texorpdfstring{$\bm{X_r^{\an}}$}{Xran}}

In \cite{JSY22} the authors define a \emph{real analytification} of a $K$-variety $X$. By a $K$-variety $X$ we always mean a reduced, irreducible, separated scheme of finite type. For a point $p$ in the scheme $X$, let $K(p)$ denote its residue field.

\begin{definition}
The \emph{real analytification} of $X$ is the set $X_r^{\an}$ consisting of all triples $x = (p_x,\lvert \cdot \rvert_x,<_x)$, where $p_x \in X, \lvert \cdot \rvert_x$ is an absolute value on $K(p_x)$ extending $\lvert \cdot \rvert_K$, and $<_x$ is an order on $K(p_x)$ compatible with $\lvert \cdot \rvert_x$. We equip $X_r^{\an}$ with the coarsest topology such that the support map \begin{align*}
\supp: X_r^{\an} &\longrightarrow X, \\ 
x &\longmapsto  p_x
\end{align*}
is continuous and the map
\begin{align*}
\supp^{-1}(U) &\longrightarrow \R, \\
x &\longmapsto \sgn_x(f) \cdot |f|_x 
\end{align*}
is continuous for every open $U \subseteq X$ and every regular function $f$ on $U$.
\end{definition}

Many properties that hold for the Berkovich analytification also hold for the real analytification. For example:
\begin{itemize}
\item For any morphism $f:X\to Y$ of varieties over $K$, there is a natural induced continuous map $f_r^{\an}: X_r^{\an}\to Y_r^{\an}$ of the corresponding real analytifications. This assignment is functorial.
\item The space $X^{\an}_r$ is a connected Hausdorff space and if $X$ is proper over $K$, it is compact \cite[Proposition 3.6, Corollary 3.10]{JSY22}. 
\end{itemize}

There is a natural map $X^{\an}_r\to X^{\an}$ which forgets the last entry. This map is always continuous and proper, but in general neither injective nor surjective \cite[Example 3.12]{JSY22}. Just like for the classical Berkovich analytification, for an affine $K$-variety $X = \Spec(A)$, the real analytification $X_r^{\an}$ has a
description as the space of signed multiplicative seminorms on the $K$-algebra $A$.

\begin{definition}
A \emph{signed multiplicative seminorm} on $A$ is a map $\sinorm{\cdot}: A \to \R$ such that:
\begin{enumerate}[(i)]
\item $|a|^{\sgn} = \sgn(a) \cdot |a|_K$ if $a \in K$,
\item  $|f \cdot g|^{\sgn} = |f|^{\sgn} \cdot |g|^{\sgn}$ if $f,g \in A,$
\item $\min(|f|^{\sgn},|g|^{\sgn})\leq |f+g|^{\sgn} \leq  \max(|f|^{\sgn},|g|^{\sgn})$ if $f,g \in A.$
\end{enumerate}
\end{definition}

The space of signed multiplicative seminorms is endowed with the coarsest topology that makes the natural evaluation maps $\ev_f: \sinorm{\cdot} \mapsto |f|^{\sgn}$ for all $f \in A$ continuous. By \cite[Proposition 3.4]{JSY22} for $X = \Spec(A)$, we have that $X_r^{\an}$ is the space of signed multiplicative seminorms on $A$. The map to the Berkovich analytification is given by $\sinorm{\cdot} \mapsto |\sinorm{\cdot}|$.

\begin{remark}
As explained in \cite{JSY22}, one can view $X_r^{\an} = X(\R\T)$, \emph{i.e.},\ as the $\R\T$-points of $X$. Namely, if $X = \Spec(A)$ is an affine $K$-scheme, then $X_r^{\an} = \Hom_K(A,\R\T)$. The latter set is the set of all hyperfield homomorphisms $A \xrightarrow{} \R\T$ that factor the canonical hyperfield homomorphism $K \xrightarrow{} \R\T, x \mapsto \sgn(x) \cdot |x|_K$ from Example \ref{exa:hyperfieldhom} (e). This construction glues to varieties over $K$. In fact, from this hyperfield point of view, \cite{jun2021geometry} showed that $X = X(\mathbb{K})$ and $X^{\an} = X(\T).$
\end{remark}

In particular, for the affine space $\A^n = \Spec K[t_1,\ldots,t_n]$, the real Berkovich analytification $(\A^n)^{\an}_r$ is given as the set of signed multiplicative seminorms on $K[t_1,\dots,t_n]$. Similar to the Berkovich Proj construction, there exists a corresponding construction for the real analytification, hence we can describe $\P_r^{n,\an}$ in terms of signed multiplicative seminorms. Let $\sinorm{\cdot}_1$ and $\sinorm{\cdot}_2$ be points in $\A^{n+1,\an}_r \setminus \{0\}$, \emph{i.e.},\ non-zero signed multiplicative seminorms on $K[t_0,\ldots,t_n]$ extending the absolute value on $K$. We call $\lvert \cdot \rvert_1^{\sgn}$ and $\lvert \cdot \rvert_2^{\sgn}$ \emph{homothetic}, if there exists a constant $0 \neq c  \in \R$ such that for every homogeneous polynomial $f$ of degree $d$, we have $|f|^{\sgn}_1 = c^d |f|^{\sgn}_2$. The following description of $\P^{n,\an}_r$ is analogous to the classical case.
\begin{proposition}
The real analytification $\P^{n,\an}_r$ is the quotient of $\A^{n+1,\an}_r \setminus \{0\}$ by homothety.
\end{proposition}

There is a natural continuous surjective real tropicalization map
\begin{align*}
\trop_r: \A^{n+1,\an}_r &\longrightarrow \R^{n+1},\\ 
\sinorm{\cdot} &\longmapsto (|t_0|^{\sgn},\dots,|t_n|^{\sgn}),
\end{align*}
which induces a real tropicalization map
\begin{align*}
\trop_{r}: \P_r^{n,\an} &\longrightarrow \R\P^n, \\
\big[ \sinorm{\cdot} \big] &\longmapsto (|t_0|^{\sgn}:\dots:|t_n|^{\sgn}).
\end{align*}

Let $\iota: X\hookrightarrow \P^n$ be a closed subvariety. By a projective version of the real analogue of the fundamental theorem \cite[Theorem 6.9]{JSY22}, we have $$\trop_{r}(X_r^{\an}) = \Trop_r(X, \iota)\, .$$

\section{The Signed Goldman--Iwahori Space}\label{section:SignedGIspace}

The signed Goldman--Iwahori space associated to a vector space can be thought of as a linear analogue of the real analytification of $\P^n$. Alternatively, it is a version of the Goldman--Iwahori space in \cite{GoldmanIwahori} that takes signs of the ground field into account. In this section, we will examine its structure and its relation to the real analytification. Further, we describe a real tropicalization map from the signed Goldman--Iwahori space to real projective space. This allows us to define a real tropicalization map from the signed Goldman--Iwahori space to real tropicalizations of linear spaces. From now on, let $K$ be a real closed field with compatible non-Archimedean absolute value $\lvert \cdot \rvert$ and let $V$ be a finite-dimensional vector space over $K$.

\subsection{Signed Seminorms}

\begin{definition}\label{def:signedSeminorm}
A \emph{signed seminorm} on $V$ is a map $\siinorm{\cdot}: V\to \R$ such that:
\begin{enumerate}[(i)]
\item For all $v\in V$ and $\lambda \in K$, we have
$$\siinorm{\lambda v}= \sgn(\lambda)|\lambda| \siinorm{v}\, .$$
\item For all $v,w \in V$ both strong triangle inequalities
$$\min(\siinorm{v},\siinorm{w})\leq \siinorm{v+w} \leq  \max(\siinorm{v},\siinorm{w})$$
hold.
\end{enumerate}
\end{definition}

\begin{remark}
This definition is equivalent to $\siinorm{\cdot} \colon V \to \R\T$ being a $\R\T$-hypervector space homomorphism.
\end{remark}

\begin{lemma}\label{lem:rulesSignedSeminorms}
Let $\siinorm{\cdot}$ be a signed seminorm on $V$ and $v,w\in V$.
\begin{enumerate}[(a)]
\item The absolute value of a signed seminorm is a seminorm (\emph{cf.}\ \cite[Definition 1.1]{BKKUV}).
\item If $\Big| \siinorm{v}\Big|>\Big|\siinorm{w}\Big|$, then $\siinorm{v+w}=\siinorm{v}$.
\item If $\Big|\siinorm{v+w}\Big|<\max\left\{\Big|\siinorm{v}\Big|,\Big|\siinorm{w}\Big|\right\}$, then $\sgn (\siinorm{v})=-\sgn(\siinorm{w})$.
\end{enumerate}
\end{lemma}

\begin{proposition}\label{prop:compSignedSeminorms}
Let $\siinorm{\cdot}_1,\siinorm{\cdot}_2$ be signed seminorms on a vector space $V$. Then the \emph{composition} 
\begin{align*}
\siinorm{\cdot}_1 \circ \siinorm{\cdot}_2: V & \longrightarrow \R, \\
v &\longmapsto \begin{cases}
\siinorm{v}_1 \quad \text{ if } \Big|\siinorm{v}_1\Big| \geq  \Big|\siinorm{v}_2\Big|\, , \\
\siinorm{v}_2 \quad \text{ else }
\end{cases}
\end{align*}
is again a signed seminorm. 
\end{proposition}
\begin{proof}
We set $\siinorm{\cdot}\coloneqq \siinorm{\cdot}_1 \circ \siinorm{\cdot}_2$. Surely, for $\lambda \in K, v \in V$, we have $\siinorm{\lambda v}=|\lambda| \siinorm{v}$. 
It remains to show that $\siinorm{v+w}\leq \max\{\siinorm{v},\siinorm{w}\}$ for all $v,w \in V$, because we can replace all involved signed seminorms by their negative. We can assume $\siinorm{v}\geq \siinorm{w}$ and $\siinorm{v}\geq 0$ after possibly switching $v,w$ and multiplying by $-1$. If $|\siinorm{w}|>|\siinorm{v}|$, by Lemma \ref{lem:rulesSignedSeminorms} (b), we have $\siinorm{v+w}=\siinorm{w}$ and hence
$$\siinorm{v+w}=\siinorm{w}\leq \max\{\siinorm{v},\siinorm{w}\}\, .$$
If $\big|\siinorm{w}\big|\leq \big|\siinorm{v}\big|$, then $$\big|\siinorm{v+w}\big|\leq \max\{\big|\siinorm{v}\big|,\big|\siinorm{w}\big|\}=\big|\siinorm{v}\big|=\siinorm{v},$$
where the first inequality follows from the fact, that the maximum of two seminorms is again a seminorm. This completes the proof.
\end{proof}

\begin{example}[Diagonalizable signed seminorms]\label{ex:diagonalizable_signed_seminorms}
Let $B= (b_1,\dots,b_n)$ be an ordered basis of $V$ and $\vec{c} = (c_1,\ldots,c_n) \in \R_{\geq 0}^{n}$ parameters such that ${c_1 \geq \dots \geq c_n}$. We may associate to this datum a map 
\begin{align*}
\siinorm{\cdot}_{B, \vec{c}}: V &\longrightarrow \R\, , \\
v=\sum_{i=0}^n \lambda_i b_i &\longmapsto 
\sgn(\lambda_j)|\lambda_j| c_j\quad \text{if } j \text{ is minimal with } |\lambda_j|c_j=\max_{i\in [n]} \{|\lambda_i|c_i\}\, .
\end{align*}
Indeed, $\siinorm{\cdot}_{B, \vec{c}}$ is a signed seminorm: Let $B_j=(b_j,b_1,\dots,\hat{b}_j,\dots,b_n)$ be a reordering of $B$ and $\vec{c}_j=(c_j,0,\dots,0)\in \R^n_{\geq 0}$. It is easy to check that $\siinorm{\cdot}_{B_j,\vec{c}_j}$ is a signed seminorm. We have 
$$\siinorm{\cdot}_{B, \vec{c}}=\siinorm{\cdot}_{B_1,\vec{c}_1}\circ\dots\circ \siinorm{\cdot}_{B_n,\vec{c}_n},$$
so by Proposition \ref{prop:compSignedSeminorms} we obtain the result. Seminorms of the form $\siinorm{\cdot}_{B, \vec{c}}$ for an ordered basis $B=(b_1,\ldots,b_n)$ and parameters $\vec{c} = (c_1,\ldots,c_n) \in \R_{\geq 0}^{n}$ such that ${c_1 \geq \dots \geq c_n}$ are called \emph{diagonalizable}.
\end{example}

\begin{example} \label{exa:ordermatters}
To illustrate that the order of the basis $B$ really matters, we consider $V = K^2$, $B = (e_1,e_2),\ B' = (e_2,e_1)$, and $\Vec{c} = (1,1)$. Then $\siinorm{e_1 - e_2}_{B,\Vec{c}} = 1$, but $\siinorm{e_1 - e_2}_{B',\Vec{c}} = - 1$.
\end{example}

In the non-signed case, if the field is spherically complete (or equivalently, maximally complete), then every seminorm on a finite-dimensional vector space is diagonalizable \cite[Proposition 1.5]{BKKUV}. In particular, this is true for trivially or discretely valued fields, since they are spherically complete. In fact, the reverse implication also holds: If $\dim V \geq 2$ and $K$ is not spherically complete, then there exists a seminorm on $V$ that is not diagonalizable. In the signed world, the situation is more complicated, as the following example will show:

\begin{example}\label{ex:not_diagonalizable}
Let $K=\R\{\!\{t\}\!\}$, $\lvert\cdot\rvert_{K}$ the $t$-adic non-Archimedean absolute value on $K$ (\emph{cf.} Example \ref{ex:puiseux_series}), and $\lvert\cdot\rvert_{\triv}$ the trivial absolute value on $K$. Consider $V=K^2$ and the map 
\begin{align*}
\siinorm{\cdot}: K^2 &\to \R \, ,\\
(x,y) &\mapsto 
\begin{cases}
\sgn(x) \quad \text{if } |x|_K \geq |y|_K\, , \\
\sgn(y) \quad \text{else.}\\
\end{cases}
\end{align*}
Then $\siinorm{\cdot}$ is a signed seminorm over $K$ with respect to $\lvert\cdot\rvert_{\triv}$ (note that we choose the trivial absolute value!)
that is not diagonalizable: Surely, $\siinorm{\lambda v}=\sgn(\lambda) |\lambda|_{\triv} \siinorm{v}$ for all $\lambda\in K,v\in V$. Checking the triangle inequalities involves checking a lot of different cases, but is entirely straightforward and will therefore be omitted. 

To see that $\siinorm{\cdot}$ is not diagonalizable, assume that there exists an ordered basis $B=(b_1,b_2)$ of $K^2$ and $c_1\geq c_2 \geq 0$ that diagonalizes $\siinorm{\cdot}$. Then $c_1=c_2=1$, since this is the only positive value that $\siinorm{\cdot}$ attains and $(0,0)$ is the only vector being mapped to $0$. For $\lambda \in K$ with $0<|\lambda|_K\ll 1$, we have
$$\siinorm{\lambda b_1+b_2}=\siinorm{b_2}=\siinorm{-\lambda b_1 + b_2},$$
which contradicts the assumption that $B,(c_1,c_2)$ diagonalize $\siinorm{\cdot}$. 

We will show in Proposition \ref{prop:norm_diagonalizable}, that if $K=\R$ with trivial valuation, then every signed seminorm is diagonalizable. It remains open, if there are other real closed fields (with trivial or non-trivial valuation), for which every signed seminorm on finite-dimensional vector spaces are diagonalizable. Our proof will make crucial use of general hyperplane separation, which is only true for $\R$, but fails for all other real closed fields \cite{robson_separating}.
\end{example}

We endow the set of signed seminorms on $V$ with the topology of pointwise convergence. This is the coarsest topology such that all evaluation maps $\ev_v:  \siinorm{\cdot} \mapsto \siinorm{v}$ for $ v \in V$ are continuous. It agrees with the subspace topology of $\R\P^{(K^{n+1})^*}$. Two signed seminorms $\siinorm{\cdot}_1, \siinorm{\cdot}_2$ are said to be \emph{homothetic}, written $\siinorm{\cdot}_1 \sim \siinorm{\cdot}_2$, if there is a constant $c \neq 0$ such that $\siinorm{\cdot}_1 = c  \siinorm{\cdot}_2$. Homothety defines an equivalence relation on the space of signed seminorms.

\begin{definition}\label{def:signedGIspace}
The \emph{signed Goldman--Iwahori space} $\RX(V)$ is defined to be the quotient of the space of non-trivial signed seminorms on the dual space $V^{\ast}$ by homothety, \emph{i.e.},
\begin{equation*}
\RX(V) = \big(\{\siinorm{\cdot}: V^{\ast} \longrightarrow \R \text{ signed seminorm} \} \setminus \{0 \} \big)/_ \sim \, .
\end{equation*}
When $V = K^{n+1}$, we write $\RX_n(K)$ for $\RX(V)$.
\end{definition}

In \cite{BKKUV}, the (non-signed) Goldman--Iwahori space is denoted by $\calXbar(V)$, which is a compactification of the space of norms modulo homothety $\calX(V)$. Even though $\RX(V)$ will turn out to be compact (Corollary \ref{cor:RX_compact}), we omit the bar to declutter notation. 

Note the dualization in Definition \ref{def:signedGIspace}. This makes the assignment $V\mapsto \RX(V)$ a covariant functor from the category of finite-dimensional vector spaces over $K$ to the category of topological spaces via pulling back signed seminorms under linear maps. For a linear map $f:V\to W$, we write $\RX(f):\RX(V)\to \RX(W)$ for the induced continuous map.

\subsection{Tropicalization of \texorpdfstring{$\bm{\RX_n(K)}$}{RXn(K)}}\label{section:trop_of_RXn}
From now on, we consider the vector space $V=K^{n+1}$ together with its defined standard basis $\mathbf{e} = (e_0,\ldots,e_n)$ and the associated dual basis $\mathbf{e}^\ast = (e_0^\ast,\ldots, e_n^\ast)$ of $V^\ast$. This identifies $\A(V)$ and $\P(V)$ with $\A^{n+1}=\Spec K[t_0,\ldots, t_n]$ and $\P^n=\Proj K[t_0,\ldots, t_n]$, respectively. 

We have a natural continuous map $\tau: \P_r^{n,\an} \xrightarrow{}  \RX_n(K)$ by restricting a non-zero signed multiplicative seminorm on $K[t_0,\dots,t_n]$ to its degree $1$ part $K[t_0,\dots,t_n]_1 \cong (K^{n+1})^{\ast}$. Note that this is well-defined, taking into account the respective equivalence relations: if $\sinorm{\cdot}_1 \sim \sinorm{\cdot}_2 \in \P_r^{n,\an}$, then there is $c\neq 0$ such that for any homogeneous linear polynomial $f \in K[t_0,\dots, t_n]_1$ we have $\sinorm{f}_1=c\sinorm{f}_2$, hence the restrictions of $\sinorm{\cdot}_1$ and $\sinorm{\cdot}_2$ are homothetic.

\begin{remark}
While in the non-signed case the natural map $\P^{n,\an}\to \calXbar_n(K)$ was proven to be surjective in \cite[Proposition 2.7]{BKKUV}, it is currently not clear to the authors, if the same holds in the signed case. However, we will show surjectivity in Proposition \ref{prop:tau_surjective}, for $K=\R$ with trivial absolute value. In that case, every signed seminorm is diagonalizable and this allows us to construct a preimage. If one could show that the subspace of diagonalizable signed seminorms of $\RX_n(K)$ is dense, then topological arguments would imply surjectivity of $\tau$.
\end{remark}

\begin{definition}\label{def:tropicalizationX_n}
The \emph{real tropicalization map} $\trop_{\RX_n}\colon \RX_n(K)\rightarrow\R\P^n$ is given by associating to a signed seminorm $\siinorm{\cdot}\colon (K^{n+1})^\ast\rightarrow \R$ the tuple 
\begin{equation*}
\trop_{\RX_n}\big(\siinorm{\cdot}\big)=\big(\siinorm{e_0^\ast} :\ldots : \siinorm{e_n^\ast} \big) \in \R\P^{n} \, .
\end{equation*}
\end{definition}

Note that the association in Definition \ref{def:tropicalizationX_n} only depends on the homothety class of $\siinorm{\cdot}$, so it indeed descends to a map $\RX_n(K)\to \R\P^n$. Moreover, by the definition of the topology of $\RX_n(K)$, the real tropicalization map $\trop_{\RX_n}$ is continuous. Using Example \ref{ex:diagonalizable_signed_seminorms}, one can construct an inverse image for each point in $\R\P^n$ after possibly replacing some basis vectors by their negatives, hence $\trop_{\RX_n}$ is also surjective.

Recall that the real tropicalization map is given by
\begin{align*}
\trop_{r}: \P^{n,\an}_r &\longrightarrow \R\P^n\, ,\\
\big[\sinorm{\cdot}\big] &\longmapsto \big(|t_0|^{\sgn}:\dots:|t_n|^{\sgn}\big)\, . 
\end{align*}
By construction, $\trop_r:\P_r^{n,\an}\to \R\P^n$ factors as 
\begin{equation*}\label{eq_factorization}    \P_r^{n,\an}\xlongrightarrow{\tau}\RX_n(K)\xlongrightarrow{\trop_{\RX_n}} \R\P^n \, .
\end{equation*}
Let $\iota = (f_0:\dots:f_m): \P^n\hookrightarrow \P^m$ be a linear embedding, where $f_i \in (K^{n+1})^*$. We define 
$$\pi_{\iota}\coloneqq \ \trop_{\RX_m}\circ\ \RX(\iota): \RX_n(K)\longrightarrow \R\P^m\, .$$
A direct computation shows that for all $\big[\siinorm{\cdot}\big] \in \RX_n(K)$, we have
\begin{equation*}\pi_{\iota}\big(\big[\siinorm{\cdot}\big]\big) = \big(||f_0||^{\sgn}:\ldots : ||f_m||^{\sgn}\big)\, .
\end{equation*}

\begin{proposition}\label{prop:tropIsProj}
For a linear embedding $\iota\colon\P^n\hookrightarrow\P^m$, we have
$$\Trop_r(\P^n,\iota)= \trop_{\RX_m} \big(\RX(\iota)\big(\RX_n(K)\big) \big).$$
In particular, the following diagram commutes:
\begin{center}
\begin{tikzcd}
\P_r^{n,\an} \ar[r,"\iota_r^{\an}"] \ar[d,"\tau"] & \P_r^{m,\an} \ar[d,"\tau"] \\
\RX_n(K) \ar[r,"\RX(\iota)"]\ar[twoheadrightarrow,d, "\pi_{\iota}"] & \RX_m(K) \ar[d,"\trop_{\RX_m}"] \\
\Trop_r(\P^n,\iota) \ar[r,"\subseteq"] & \R\P^m
\end{tikzcd}
\end{center}

\end{proposition}
\begin{proof}
Let $\mathcal{M}$ be the realizable oriented valuated matroid on $\{f_0,\ldots,f_m\}\subset (K^{n+1})^*$ associated to the embedding $\iota$. By Proposition \ref{prop:tableratropicalbasis}, $\Trop_r(\P^n,\iota)$ equals the intersection of the real tropical hypersurfaces $\Trop_r(\calM_C)$, where $C$ ranges over the signed valuated circuits of $\calM$. Let $\siinorm{\cdot}$ be a signed seminorm on $(K^{n+1})^*$ and let $C$ be a signed valuated circuit of $\calM$. Then $C$ is associated to a minimal linear dependence $\sum_i \lambda_i f_i = 0$ among the $f_i$, \emph{i.e.}, $C = \{(\sgn(\lambda_1) |\lambda_1|,\ldots,\sgn(\lambda_m) |\lambda_m| \}$.  We need to show that $(\siinorm{f_0}:\dots:\siinorm{f_m})$ lies in 
\begin{align*}
\Trop_r(\calM_C) &= \{(y_0:\ldots:y_m) \in {\R\P}^m: \text{ there exist indices } i \neq j  \text{ such that} \\
&\quad \max_{e \in \Supp C}(|y_e|\cdot |C_e|) \text{ is attained at } i,j \text{ and } y_i \cdot C_i = - y_j C_j \}.
\end{align*}
The strong triangle inequalities for signed seminorms imply that 
$$\textstyle \big| \siinorm{ \sum_i \lambda_i f_i}\big| \leq \max_i(|\lambda_i| \siinorm{f_i}|) = \max_{e \in \Supp C}(|C_e \cdot \siinorm{f_e}|)\, .$$ 
Since  $\sum_i \lambda_i f_i = 0$, the maximum is attained twice. Let $I \subseteq \Supp C$ be the set of indices where the maximum $\max_{e \in \Supp C}(|C_e \cdot \siinorm{f_e}|)$ is attained. Again by the triangle inequalities applied to $\siinorm{ \sum_{i \in I} \lambda_i f_i}$ it follows from Lemma \ref{lem:rulesSignedSeminorms} that there must exist indices $i,j \in I, i \neq j$ with $ \siinorm{f_i} \cdot C_i = - \siinorm{f_j} C_j$.
\end{proof}

\section{The Limit Theorem}\label{section:limitthm}

In this section, we show a linear version of the limit theorem \cite[Theorem 6.14]{JSY22} which is a real analogue of the limit theorem \cite[Theorem A]{BKKUV}. We will have a similar setup as in \cite[\S 3]{BKKUV} which we now recall. We first set up a category of linear embeddings such that real tropicalization yields a covariant functor into the category of topological spaces. 

\begin{definition}\label{def:category}
Let $I$ be the cofiltered category of linear embeddings $\P^n \hookrightarrow U \subseteq \P^m$, where $U$ is a torus-invariant open subset of $\P^m$ with morphisms given by commutative triangles 
\begin{center}
\begin{tikzcd}
\P^m \ar[hookrightarrow,r,"\iota"] \ar[hookrightarrow,dr,"\iota'"'] & U \ar[d] \\
& U'
\end{tikzcd}
\end{center}
where $U\to U'$ is a composition of a coordinate projection and a coordinate permutation.
\end{definition}

\begin{lemma}\label{lem:well-definedness}
Let $\iota:\P^n\hookrightarrow U\subseteq \P^m$ and $ \iota':\P^n\hookrightarrow U'\subseteq \P^{m'}$ be linear embeddings and $\varphi: U\to U'$ be a morphism in $I$ with $\varphi\circ \iota = \iota'$. 
\begin{enumerate}[(a)]
\item The morphism $\varphi$ induces a natural composition of coordinate projections and permutations \[\varphi_r^{\trop}: \Trop_r(\P^n,\iota)\to  \Trop_r(\P^n,\iota')\, .\] 
\item The following diagram commutes: 
\begin{center}
\begin{tikzcd}
\RX_n(K) \ar[r, "\pi_{\iota}"] \ar[dr, "\pi_{\iota'}"'] & \Trop_r(\P^n,\iota) \ar[d, "\varphi_r^{\trop}"] \\
& \Trop_r(\P^n,\iota')
\end{tikzcd}
\end{center}
In particular, there is a natural map $$\varprojlim_{\iota \in I} \pi_{\iota}: \RX_n(K) \to \varprojlim_{\iota \in I} \Trop_r\big(\PP^n,\iota\big).$$
\end{enumerate}
    
\end{lemma}

\begin{theorem}\label{thm:limits}
The natural map
$$\varprojlim_{\iota \in I} \pi_{\iota}\colon \RX_n(K)\stackrel{\cong}{\longrightarrow} \varprojlim_{\iota \in I} \Trop_r\big(\PP^n,\iota\big)$$ 
from Lemma \ref{lem:well-definedness} is a homeomorphism.
\end{theorem}
\begin{proof}
The proof is a combination of \cite[Theorem 1.1]{payne_analytification}, \cite[Theorem 6.13]{JSY22} and \cite[Theorem 3.5]{BKKUV}.

To show injectivity of $\varprojlim_{\iota \in I} \pi_{\iota}$, let $[\siinorm{\cdot}_1], [\siinorm{\cdot}_2] \in \RX_n(K)$ be two homothety classes such that $\varprojlim_{\iota \in I} \pi_{\iota} ([\siinorm{\cdot}_1]) = \varprojlim_{\iota \in I} \pi_{\iota} ([\siinorm{\cdot}_2]).$ By the same argument as in the proof of \cite[Theorem 3.5]{BKKUV} the signed seminorms $\siinorm{\cdot}_1, \siinorm{\cdot}_2$ have the same kernel. Let $f,g \in (K^{n+1})^*$ be outside of this kernel. We extend $f,g$ to a generating set $f,g,f_2,\ldots,f_m$ of $(K^{n+1})^*$ and we consider the corresponding linear embedding $\iota = [f:g:f_2:\cdots:f_m]: \P^n \hookrightarrow \P^m$. Then we have $\pi_\iota([\siinorm{\cdot}_1]) = \pi_\iota([\siinorm{\cdot}_2])$ as elements in $\R\P^m$ and thus the quotient of the two first coordinates equals $$\frac{\siinorm{f}_1}{\siinorm{g}_1} = \frac{\siinorm{f}_2}{\siinorm{g}_2}. $$
This implies that $\siinorm{\cdot}_1$ and $\siinorm{\cdot}_2$ are homothetic.

To show surjectivity of $\varprojlim_{\iota \in I} \pi_{\iota}$, let $(y_{\jmath})_{\jmath \in I}\in \varprojlim_{\jmath\in I} \Trop_r\big(\PP^n,\jmath\big).$ First, we consider the identity $\id=\big[e_0^{\ast}:\ldots:e_n^{\ast}\big]:\P^n\to\P^n$. After a permutation of coordinates we may assume that the first coordinate $y_{\id,0}$ of $y_{\id}\in \R\P^n$ is not $0$. As in the proof of \cite[Theorem 3.5]{BKKUV} it follows that for all linear embeddings $\iota=\big[e_0^{\ast}:f_1:\cdots:f_{m}\big]$, the first coordinate $y_{\iota,0}$ is not $0$.  We will construct a signed seminorm $\siinorm{\cdot}$ with $\siinorm{e_0^{\ast}}=1$ and $\pi_{\jmath}(\siinorm{\cdot})=y_{\jmath}$ for all $\jmath\in I$. Let $f \in (K^{n+1})^*.$ We choose an embedding $\jmath=\big[e_0^*:f:f_2:\cdots:f_m \big]: \P^n \to \P^m$ and define
$$\siinorm{f} := \frac{y_{\jmath,1}}{y_{\jmath,0}}.$$  By the same argument as in the proof of \cite[Theorem 3.5]{BKKUV} using coordinate projections and permutations, this definition does not depend on the choice of $\jmath.$

We check that the constructed map is indeed a signed seminorm. For $f \in (K^{n+1})^*$ and $\lambda \in K$ consider any embedding $\jmath=\big[ e_0^{\ast}:f:\lambda f: \ldots \big]$. Then, by Proposition \ref{prop:tropIsProj}, for every $y_{\jmath}\in \Trop_r(\P^n,\jmath)$ there is a class of a signed seminorm $\big[\siinorm{\cdot}_1\big] \in \RX_n(K)$ with 
        \begin{align*}
            y_{\jmath}&=\pi_{\jmath}\big(\big[\siinorm{\cdot}_1\big]\big) \\
            &= \big(\trop_{\RX_m}\circ \RX(\jmath)\big)\big(\big[\siinorm{\cdot}_1\big]\big)\\
            &=\big[\siinorm{e_0^{\ast}}_1: \siinorm{f}_1: \siinorm{\lambda f}_1: \cdots \big]
        \end{align*} 
and thus $\sgn(\lambda) |\lambda| y_{\jmath,1} =y_{\jmath,2}.$ Therefore, $$\siinorm{\lambda f} = \frac{y_{\jmath,2}}{y_{\jmath,0}} = \frac{\sgn(\lambda) |\lambda| y_{\jmath,1}}{y_{\jmath,0}} = \sgn(\lambda) |\lambda| \siinorm{f}.$$ For $f,g \in (K^{n+1})^{\ast}$, the inequalities $ \min\{\siinorm{f},\siinorm{g}\} \leq  ||f+g||^{\sgn}\leq \max\{\siinorm{f},\siinorm{g}\}$ follow similarly by considering an embedding containing $f,g$ and $f+g$. By construction, the signed seminorm $\siinorm{\cdot}$ is an inverse image of $(y_{\jmath})_{\jmath\in I}$.

Finally, as in the proof of \cite[Theorem 6.13]{JSY22}, the map is a homeomorphism because the topology on the left is defined as the coarsest topology such that $\siinorm{\cdot} \mapsto \siinorm{f}$ is continuous for all $f \in (K^{n+1})^*$, while on the right the topology is defined such that all projection maps, that is, all maps $\varprojlim_{\iota \in I} \Trop_r\big(\PP^n,\iota\big) \to \Trop_r\big(\PP^n,\iota\big)$ to a particular $\iota$ are continuous. These conditions are equivalent.
\end{proof}

\begin{corollary}\label{cor:RX_compact}
The signed Goldman--Iwahori space $\RX_n(K)$ is compact and Hausdorff.
\end{corollary}
\begin{proof}
This follows from Theorem \ref{thm:limits} since all $\Trop_r(\P^n,\iota)$ are compact and Hausdorff, hence the inverse limit is a closed subspace of the compact product space and is thus itself compact and Hausdorff.
\end{proof}

\begin{remark}
Let $K$ be trivially valued. Then all real tropicalized linear spaces $\Trop_r(\P^n,\iota)$ are homeomorphic to $\R\P^n$, see Corollary \ref{cor:BMfanIsTrop} and Proposition \ref{prop:realtrophomeom}. One might suspect that the limit space $\RX_n(K) $ is also homeomorphic to $\R\P^n$. However, in the category $I$ there are coordinate projections, which do not induce homeomorphisms between the respective real tropicalizations. Hence, one cannot conclude that $\RX_n(K) $ is homeomorphic to $\R\P^n$.
\end{remark}

\subsection{Relation to the Goldman--Iwahori Space}\label{section:relGIspace}
Recall that the \emph{Goldman--Iwahori space} $\calXbar_n(K)$ is the space of non-trivial (ordinary) seminorms on $(K^{n+1})^*$ modulo homothety. It has been studied first by Goldman and Iwahori \cite{GoldmanIwahori} (also see for more recent accounts \cite{Werner_seminorms,RemyThuillierWernerII, RemyThuillierWerner_survey}). However, our notation will follow \cite{BKKUV}, where it was proved that $\calXbar_n(K)$ is the limit of all linear tropicalizations of $\P^n$. The \emph{compactified affine Bruhat--Tits building} $\calBbar_n(K)$ of $\PGL((K^{n+1})^*)$ is the subspace of $\calXbar_n(K)$ of classes of \emph{diagonalizable} seminorms. The inclusion $\calBbar_n(K)\subseteq \calXbar_n(K)$ is dense by \cite[Remark 1.9]{BKKUV} and we have an equality if and only if $K$ is \emph{spherically complete} by \cite[Satz 24]{Krull}. 

The association $\siinorm{\cdot} \mapsto \big|\siinorm{\cdot}\big|$ defines a natural map 
$$\Phi: \RX_n(K) \longrightarrow \calXbar_n(K)\, .$$
Since morphisms in $I$ commute with taking pointwise absolute values, we obtain the following:

\begin{theorem}\label{thm:commutativityPhi}
\ 
For any morphism $\varphi \in I$ between linear embeddings $\iota,\iota'$ of $\P^n$, the following diagram commutes:
\begin{center}
\begin{tikzcd}
\RX_n(K) \ar[r,"\pi_{\iota}"] \ar[d, "\Phi"] \arrow[bend left=20]{rr}{\pi_{\iota'}}
 & \Trop_r(\P^n,\iota)\ar[r, "\varphi_r^{\trop}"] \ar[d] & \Trop_r(\P^n,\iota') \ar[d] \\
\calXbar_n(K) \ar[r,"\pi_{\iota}"]  \arrow[bend right=20]{rr}{\pi_{\iota'}} & \Trop(\P^n,\iota) \ar[r, "\varphi^{\trop}"] & \Trop(\P^n,\iota')
\end{tikzcd}
\end{center}
In particular, $\Phi$ is exactly the map induced by all $|\pi_{\iota}|:\RX_n(K)\to \Trop(\P^n,\iota)$.
\end{theorem}

\begin{proposition}
The map $\Phi: \RX_n(K) \longrightarrow \calXbar_n(K)$ is surjective.
\end{proposition}
\begin{proof}
The Goldman--Iwahori space is a compact Hausdorff space, since it is the limit of compact Hausdorff spaces. Moreover, so is the signed Goldman--Iwahori space by Corollary \ref{cor:RX_compact}. Via Example \ref{ex:diagonalizable_signed_seminorms}, we construct preimages for all classes $\big[||\cdot||\big]\in \calBbar_n(K)$, hence the image of $\Phi$ contains $\calBbar_n(K)$. Since the image of $\Phi$ is compact, it needs to be closed, thus density of $\calBbar_n(K)$ implies surjectivity of $\Phi$.
\end{proof}

In general, the fibers of $\Phi$, though non-empty, can vastly differ in complexity. In Section \ref{section:Rtrivial}, we will consider the trivially valued case $K=\R$ and show that the fibers can be singletons, finite sets, or even infinite.

\begin{remark}[Generalization to enriched valuations]\label{rem:enriched}
The limit theorem (Theorem~\ref{mainthm:limits}) fits into a broader framework of \emph{enriched valuations} \cite{maxwell2023geometry, maxwell2024generalising}. 
For any enriched valuation $v: K \to \mathbb{H}[\R]$ one defines $v$-seminorms on $(K^{n+1})^*$ and a \emph{$v$-Goldman--Iwahori space} $\mathcal{X}_n(K,v)$ of non-trivial $v$-seminorms modulo $(\mathbb{H}[\R])^*$. Provided that the $v$-tropicalization of a linear embedding $\iota\colon \P^n \hookrightarrow \P^m$ is cut out by the $\mathbb{H}[\R]$-circuits of the associated $\mathbb{H}[\R]$-matroid, the same proof as Theorem~\ref{mainthm:limits} yields an isomorphism
\[
\mathcal{X}_n(K,v) \stackrel{\cong}{\longrightarrow} \varprojlim_{\iota \in I} \Trop_v\!\left(\P^n, \iota\right).
\]
For example, this property holds for the fine valuation, see \cite{smith2024matroidstropicalextensionstracts}[Prop. 5.12]. 
It would be interesting to develop a universal $\mathbb{H}[\R]$-matroid description of $\mathcal{X}_n(K,v)$ analogous to Theorem~\ref{mainthm:GIspaceIsTropMuniv}, and to study the corresponding $v$-Goldman--Iwahori space as a universal enriched tropical linear space.
\end{remark}

\section{The Universal Realizable Oriented Valuated Matroid}\label{section:Muniv}
We may extend the definition of a matroid over a hyperfield (Definition \ref{def:matroidsoverhyperfieldsGP}) to infinite ground sets. This generalization to infinite ground sets works as in \cite[\S 4, \S 7]{BKKUV} which is done for the tropical hyperfield $\T$, \emph{i.e.}, for valuated matroids of finite rank on possibly infinite ground sets. In this section, we will do the same for the real tropical hyperfield $\R\T$.

Let $E$ be now a possibly infinite ground set and $\mathcal{M}$ be an oriented valuated matroid given by a Grassmann--Plücker function $\varphi: E^n \to \R\T$. Note that the rank of $\mathcal{M}$ is $n$, so it is in particular finite.  We define $\R\P^E : = \{(y_e)_{e\in E} | y_e \in \R \} \setminus \{(0)_{e \in E}\}  / \R^*$ and equip it with the topology of pointwise convergence. Analogous to Definition \ref{def:realtropicallinearspace} we define: 

\begin{definition}
The \emph{real tropical linear space} $\Trop_r(\mathcal{M}) \subseteq\R\P^E$ associated to $\mathcal{M}$ is the set of $(y_e)_{e \in E} \in \R\P^E$ such that for any $C=(f_0,\dots,f_{n+1}) \in E^{n+2}$ we have that 
$$0 \in \bigoplus_{i=0}^{n+1} (-1)^i y_{f_i} \cdot \varphi(C - f_i)\, .$$
\end{definition}

Recall that $\oplus$ denotes the hyperfield sum of the real tropical hyperfield $\R\T=\R$. 

Let $K$ be a real closed field with a compatible absolute value $\lvert \cdot \rvert_K$. We now extend the construction of realizable oriented valuated matroids from Example \ref{example:realizableorvalmatroid} to the ground set $E = (K^{n+1})^*$. 

\begin{definition}
Let $E = (K^{n+1})^*$. The \emph{universal realizable oriented valuated matroid} or \emph{universal $\R\T$}-matroid $\mathcal{M}_{\univ}$ of rank $n+1$ is given by the Grassmann--Plücker function
\begin{equation*}\begin{split}
\varphi_{\univ}:\ E^{n+1} &\longrightarrow \R\T, \\
(f_0,\ldots,f_n) &\longmapsto \sgn(\det[f_0,\dots,f_n]) \cdot \lvert\det[f_0,\dots,f_n]\rvert\, .
\end{split}\end{equation*}  
\end{definition}

By considering $- \log |\varphi_{\univ}|$, we obtain the universal realizable valuated matroid $w_{\univ} = \val \circ \det$ studied in \cite{BKKUV}.  Hence, in contrast to $w_{\univ}$, the map $\varphi_{\univ}$ also keeps track of the signs.

Note that both the signed Goldman--Iwahori space and the real tropical linear space associated to the universal realizable oriented valuated matroid are defined to be subsets of $\R\P^E$.

\begin{theorem}\label{thm:GIspaceIsTropMuniv}
The signed Goldman--Iwahori space is the real tropical linear space associated to the universal realizable oriented valuated matroid $\mathcal{M}_{\univ}$, \emph{i.e.}
\begin{align*}
\RX_n(K) = \Trop_r(\mathcal{M}_{\univ}).
\end{align*}
\end{theorem}
\begin{proof}
Consider a signed seminorm $\siinorm{.}$ on $(K^{n+1})^*$. To show that $\big[\siinorm{.}\big] \in \Trop_r(\mathcal{M}_{\univ})$, take $C=(f_0,\dots,f_{n+1}) \in E^{n+2}$. Then, by an application of Cramer's rule 
$$\sum_i (-1)^i\det(C-f_i) f_i = 0\, .$$
We denote by $\lambda_i \coloneqq (-1)^i\det(C-f_i)$. Then
$$0 = \siinorm{ \sum_{i} \lambda_i f_i } \in \bigoplus_{i} \sgn(\lambda_i)|\lambda_i| \cdot \siinorm{f_i}.$$ 
By construction, $\sgn(\lambda_i)|\lambda_i| = (-1)^i \varphi_{\univ}(C - f_i)$. This shows that $\big[\siinorm{.}\big] \in \Trop_r(\mathcal{M}_{\univ}).$

Conversely, using similar methods as in \cite[\S 7.1]{BKKUV}, one can see that the circuit conditions of $(y_e)_{e \in E}$ being in $\Trop_r(\mathcal{M}_{\univ})$ imply that $\siinorm{e} = y_e$ is a signed seminorm.
\end{proof}

\begin{remark}\label{rem:GIspaceIsTropMuniv}
\begin{enumerate}[(a)]
\item Theorem \ref{thm:GIspaceIsTropMuniv} is a signed analogue of \cite[Theorem C]{BKKUV} which gives an identification $\calXbar_n(K) = \Trop(w_{\univ})$.
\item Theorem \ref{thm:GIspaceIsTropMuniv} has a compelling interpretation in light of the limit Theorem \ref{thm:limits}. For any finite subset $E' \subseteq E$ containing a basis, we can restrict $\mathcal{M}_{\univ}$ to $E'$ and we have a natural surjective map $\Trop_r(\mathcal{M}_{\univ}) \xrightarrow{} \Trop_r(\mathcal{M}_{\univ}|_{E'})$. We have $$\Trop_r(\mathcal{M}_{\univ}) \cong \varprojlim_{\substack{E' \subset E \\ |E'| < \infty } } \Trop_r(\mathcal{M}_{\univ}|_{E'})$$ 
where the limit is taken over all finite generating subsets $E'$ of $E$. Hence, one could alternatively prove Theorem \ref{thm:GIspaceIsTropMuniv} via Theorem \ref{mainthm:limits} and Proposition \ref{prop:tableratropicalbasis}.
\item Theorem \ref{thm:GIspaceIsTropMuniv} gives an interpretation of $\RX_n(K)$ as a real tropicalized linear space in the following way: Consider for $E=(K^{n+1})^*$ the \emph{universal embedding}
\[\iota_{\univ}: \P^{n}\longhookrightarrow \P\left( K^{E}\right) \, .\] 
Then the associated oriented valuated matroid is of course $\mathcal{M}_{\iota_{\univ}}=\mathcal{M}_{\univ}$. Hence, $\RX_n(K)$ can be seen as the real tropicalization of $\P^n$ with respect to $\iota_{\univ}$.
\end{enumerate}
\end{remark}


\section{The Case of Real Numbers with Trivial Valuation}\label{section:Rtrivial}

For the case that $K=\R$ with trivial valuation, it will turn out that all signed seminorms are diagonalizable. This lets us describe the signed Goldman--Iwahori space explicitly over this field. Moreover, it allows us to describe the real tropical linear space of $\Muniv$ as a \emph{real Bergman fan}.

\subsection{Diagonalizability of Signed Seminorms}\label{section:diag_signed_seminorms}
We want to describe $\RX(V)$ explicitly by making use of the map $\Phi: \RX(V)\to \calBbar(V)$, which assigns to each signed seminorm its absolute value. Therefore, we quickly recall an explicit description of $\calBbar(V)$ in terms of flags of subspaces. The following holds for any trivially valued field $K$:

\begin{proposition}[{\cite[Example 1.10]{BKKUV}}]\label{prop:building_trivial_valuation}
Let the dimension of $V$ be $n$.
There is a bijection
$$\calBbar(V) \stackrel{1:1}{\longleftrightarrow} \big\{\big(0 = V_0 \subsetneq  V_1 \subsetneq  \dots \subsetneq  V_l=V^{\ast},0< d_1 < \dots < d_{l-1} <1 \big), d_i \in \R \big\}_{l=1,\ldots,n}\, .$$
In other words, a seminorm is given by a flag of subspaces together with increasing real weights.
\end{proposition}
The bijection works the following way: for a homothety class of seminorms we choose the representative $\lVert \cdot \rVert$ that has maximal value $1$ on $V$. Then one obtains the flag of subspaces as subsets $\lVert \cdot \rVert^{-1}(0,\varepsilon)$ by letting $\varepsilon$ vary. The kernel of $\lVert \cdot \rVert$ then equals $V_0$ and the coordinates $d_i$ are given by the constant value of $\lVert \cdot \rVert$ on $V_i\setminus V_{i-1}$ for $i=1,\dots,l-1$. In particular, $\calBbar(V)$ can be identified with the compactified cone over the order complex of the lattice of non-trivial subspaces of $V^{\ast}$.

\begin{minipage}[l]{.56 \textwidth}
\begin{example}\label{example:calBbar(K^2)}
In the building $\calBbar_1(K)$ homothety classes of seminorms correspond to flags of subspaces of $(K^2)^{\ast}$ together with a single coordinate $0 < d < 1$. 
Each cone corresponds to a one-dimensional subspace $V_1$ and the point at infinity of each cone corresponds to the homothety class of a proper seminorm, as visualized in Figure \ref{fig:example_trivial_valuation}. A norm in the homothety class corresponding to $(V_1,d)$ has generic value $1$, and value $d$ on $V_1\setminus\{0\}$. A boundary point is given by a subspace $V_0$ which is the kernel of a proper seminorm. Hence in this particular case of rank one, the boundary equals $\P^1(K)$. The central point $\eta$ corresponds to the class of the seminorm that takes value $1$ everywhere except at $0$.
\end{example}
\end{minipage}
\hspace{-1cm}
\begin{minipage}[c]{.6 \textwidth}
\begin{center}
\begin{tikzpicture}[scale=0.65]

\draw[black, dotted, thick](0,0) circle (3) {};

\draw[] (-2.5,0) -- (2.5,0);
\draw[dotted] (-3,0) -- (3,0);
\draw[] (0,-2.5) -- (0,2.5);
\draw[dotted] (0,-3) -- (0,3);
\draw[] ({-sqrt(6.25/2)},{-sqrt(6.25/2)}) -- ({sqrt(6.25/2)},{sqrt(6.25/2)});
\draw[dotted] ({-sqrt(9/2)},{-sqrt(9/2)}) -- ({sqrt(9/2)},{sqrt(9/2)});
\draw[] (0,0) -- ({sqrt(6.25/2)},{-sqrt(6.25/2)});
\draw[dotted] (0,0) -- ({sqrt(9/2)},{-sqrt(9/2)});

\draw[very thick, loosely dotted] (-0.5,2) arc (atan(1/4)+90:180-atan(1/4):2) ;

\draw[] (3,0.2) node[anchor=south west]{$\P^1(K)$};
\filldraw[black] (0,0) circle (2pt) node[anchor=south east]{$\eta$};
\filldraw[black] ({sqrt(9/2)},{sqrt(9/2)}) circle (2pt) node[anchor=south west]{$V_0$};
\filldraw[black] (1,1) circle (2pt) node[anchor=north west]{$(V_1,d )$};
\draw[white] (-3,0.2) node[anchor=south east]{$.$};

\end{tikzpicture}
\captionof{figure}{The building $\calBbar_1(K)$}
\label{fig:example_trivial_valuation}
\end{center}

\end{minipage}

\vspace{3mm}
From now on, we again consider $K=\R$. Preimages of intervals in the signed case are no longer subspaces, but only convex cones:

\begin{lemma}\label{lem:convex_cone}
Let $\siinorm{\cdot}$ be a signed seminorm on $V$ and $[a,b]$ be a closed interval in $\R$. Then $L = (\siinorm{\cdot})^{-1}([a,b])\cup \{0\}$ is a convex cone in $V$. Moreover, if $b \geq 0$ and $a=-b$, then $L$ is a subspace.
\end{lemma}
\begin{proof}
Surely for $\lambda \in \R_{> 0}$ and $v\in V$, we have $||\lambda v||^{\sgn}=||v||^{\sgn}$, so the preimage of an interval is a cone. To see convexity, let $v,w \in L, t \in [0,1]$. Then $$||(1-t)v + tw|| \leq \max \{||(1-t)v||, ||tw||\} = \max\{||v||, ||w||\} \leq b.$$ Similarly, $$||(1-t)v + tw|| \geq \min\{||(1-t)v||, ||tw||\} = \min\{||v||, ||w||\} \geq a.$$ If $b\geq 0$ and $a=-b$, then surely also $v\in L$ implies $-w \in L$, so with the above $L$ is also a subspace.
\end{proof}

\begin{lemma}\label{lem:properties_siinorms}
Let $0\neq \siinorm{\cdot}$ be a signed seminorm on $V$. 
\begin{enumerate}[(a)]
\item The image of $\siinorm{\cdot}$ is finite.
\item For $a=\max_{v\in V} ||v||^{\sgn}$, we have that $A\coloneqq \{v \in V \mid ||v||^{\sgn}=a\}$ is a convex cone. 
\item There exists $w \in V^{\ast}$ such that 
\begin{align*}
 A^{\circ} &= \{v \in V \mid \langle v,w \rangle >0\}, \\
 \overline{A} &= \{v \in V \mid \langle v,w \rangle \geq 0\}\text{ and } \\
 \partial A &=w^{\perp}   
\end{align*}

where $A^{\circ}/\overline{A}/\partial A$ denote the interior/closure/boundary of $A$ in $V$ with respect to the Euclidean topology.
\item A signed seminorm $\siinorm{\cdot}$ induces a well-defined signed seminorm on the quotient $V/(\ker \siinorm{\cdot}) \to \R$.
\item The restriction of $\siinorm{\cdot}$ to any subspace is again a signed seminorm.
\end{enumerate}
\end{lemma}
\begin{proof}
\begin{enumerate}[(a)]
\item Follows immediately from Proposition \ref{prop:building_trivial_valuation}.
\item Follows from Lemma \ref{lem:convex_cone} applied to the interval $[a-\varepsilon, a+\varepsilon ]$ for a sufficiently small $\varepsilon>0$.
\item The second equality of sets follows from the first one. Since $A$ is convex, so is $-A$. By the separation Lemma \cite[3.3.9]{stoer2012convexity}, there is a hyperplane $H$ separating $A$ and $-A$. Let $w\in V^{\ast}$ such that $H= w^{\perp}$ and $A \subseteq \{v \in V \mid \langle v,w \rangle \geq 0\}$. Then $A^{\circ} \subseteq \{v \in V \mid \langle v,w \rangle > 0\}$. From Proposition \ref{prop:building_trivial_valuation}, one can deduce the other inclusion.
\end{enumerate}
We omit the proofs of (d) and (e) as these are straightforward.
\end{proof}

We note that $K = \R$ is the only real closed field that admits the separation Lemma \cite[3.3.9]{stoer2012convexity}. Thus for an arbitrary real closed field $K$ we cannot expect a hyperplane $\langle \cdot,w \rangle = 0$ separating $A$ and $-A$ to exist.

This now allows us to proof, that every signed seminorm is diagonalizable:

\begin{proposition}\label{prop:norm_diagonalizable}
Let $K=\R$. Every signed seminorm $\siinorm{\cdot}$ on $V$ is of the form $\siinorm{\cdot}_{B,\vec{c}}$ for an ordered basis $B = (b_1,\dots,b_n)$ and parameters $c_1 \geq c_2 \geq \dots \geq c_n \in \R_{\geq 0}$.
\end{proposition}
\begin{proof}
The trivial signed seminorm is diagonalizable by any basis setting $\vec{c}=0$. Let $\siinorm{\cdot}$ be non-trivial. As in Lemma \ref{lem:properties_siinorms}, let $a=\max_{v\in V} ||a||^{\sgn}$, $A\coloneqq \{v \in V \mid ||v||^{\sgn}=a\}$, and $H=\partial A$. Set $b_1$ to be any vector in $A^{\circ}$ and $c_1\coloneqq a$. By Lemma \ref{lem:properties_siinorms} (d), the restriction $\siinorm{\cdot} |_H$ is a signed seminorm and therefore, by induction, it is diagonalizable with $b_2,\dots,b_n \in H$, and $c_2,\dots,c_n \in \R$ with $c_2 \geq \dots \geq c_n$. Then $(b_1,\dots,b_n)$ is a basis of $V$, by definition $c_1 \geq c_2 \geq \dots \geq c_n$, and for any $v=\sum_{i=1}^n \lambda_i b_i$ we have that $v \in H$ if and only if $\lambda_1=0$. Moreover, we have $v \in A$ if and only if $\lambda_1>0$, and $v \in -A$ if and only if $\lambda_1<0$. In all three cases we have $\siinorm{v}=\siinorm{v}_{B,c}$.
\end{proof}

\begin{remark}
The key ingredient for Proposition \ref{prop:norm_diagonalizable} is hyperplane separation for general convex sets, which, as noted above, only holds over $\R$. It is currently not clear to the authors, if there are other real fields, trivially or non-trivially valued, for which every signed seminorm is diagonalizable.
\end{remark}

Diagonalizability gives us a description of $\RX(V)$ in the flavor of Proposition \ref{prop:building_trivial_valuation}.

\begin{definition}
A \emph{signed flag} of subspaces is given by the data of
\begin{itemize}
\item a flag $0\subseteq V_0 \subsetneq V_1 \subsetneq \dots \subsetneq V_l=V$ of subspaces such that for all $i=1,\dots,l$ we have $\dim V_{i-1} = \dim V_{i}-1$, and
\item a choice of a convex region in $V_i\setminus V_{i-1}$ for all $i=1,\dots, l$.
\end{itemize}
\end{definition}
We consider two signed flags \emph{equivalent}, if the underlying flag of subspaces is the same and the choice of region is exactly opposite for all $i$. The notion of signed flags helps us to formulate a signed analogue of Proposition \ref{prop:building_trivial_valuation}.

\begin{proposition}\label{prop:normsFromFlags}
Let $V$ be of dimension $n$. There is a bijection
\begin{align*}
\RX(V) \stackrel{1:1}{\longleftrightarrow} \big\{\big(0\subseteq V_0 \subsetneq V_1 \subsetneq \dots \subsetneq V_l=V^*,\ 0<d_1\leq \dots\leq d_{l-1} \leq 1\big)\big\},
\end{align*}
where the right hand side runs over all equivalence classes of signed flags.
\end{proposition}
\begin{proof}
For a point in $\RX(V)$, we first choose a representative $\siinorm{\cdot}$ such that its maximal value is $1$. Now, by Proposition \ref{prop:norm_diagonalizable} $\siinorm{\cdot}=\siinorm{\cdot}_{B,\vec{c}}$ for an ordered basis $B=(b_1,\dots,b_n)$ and $\vec{c}=(c_1=1,c_2,\dots,c_n)\in \R_{\geq 0}^n$ with $c_i\geq c_{i-1}$ for all $i$. Let $i_0 = \min\{i=1,\dots,n+1 \mid c_i=0\}$, where we formally set $c_{n+1}=0$. We define a signed flag by setting $V_0=\langle b_{i_0},\dots,b_n\rangle$, $V_i=\langle b_{i_0-i},\dots, b_n\rangle$, and chose the region of $V_i$ that contains $b_i$. Note that both the choice of the other representative or of a different diagonalizing basis would yield an equivalent signed flag. Moreover, $\vec{c}$ is independent of both choices. Finally, we set $d_i \coloneqq  c_{i_0-i}$ for $i=1,\dots,i_0-2$ to obtain the desired data. 

The above procedure is reversible by successively constructing a basis and the homothety class of the resulting signed seminorm does not depend on the choice of the constructed basis.
\end{proof}

Together, Proposition \ref{prop:building_trivial_valuation} and Proposition \ref{prop:normsFromFlags} allow us to study the natural map $\Phi: \RX(V)\to \calBbar(V)$. We want to understand $\Phi([\siinorm{\cdot}])$ for some $[\siinorm{\cdot}]\in \RX(V)$. Let $0\subseteq V_0 \subsetneq V_1 \subsetneq \dots \subsetneq V_l=V^*$ be the corresponding signed flag and $0\leq d_1\leq \dots\leq d_{l-1} \leq 1$ constructed as in Proposition \ref{prop:normsFromFlags}. Let $i_0\coloneqq  0$ and $0<i_1\leq \dots \leq i_{l'} <l$ be the indices such that there is a strict inequality $d_{i_j}<d_{i_j+1}$ (here we set and $d_l \coloneqq  1$). Set $V'_j\coloneqq  V_{i_j}$ for $j=1,\dots,l$ and $d'_j=d_{i_j}$. Then $\Phi\left(\big[\siinorm{\cdot}\big]\right)$ is given by the flag $0=V'_0 \subsetneq V'_1 \subsetneq \dots \subsetneq V'_{l'} \subsetneq V'_{l'+1} \coloneqq  V^*$ and coordinates $0\leq d'_1<\dots< d'_l<1$ as in Proposition \ref{prop:building_trivial_valuation}.

A natural question to ask is: what do the fibers of $\Phi$ look like? This has an easy answer, if the flag that is induced by the homothety class of a seminorm is complete, \emph{i.e.}\ jumps of subspaces are only by one dimension. In that case, all preimages are given by the same coordinates $d_i$ and any choice of signature on the flag yields a signed seminorm, hence the fiber under $\Phi$ has exactly $2^{n-1}$ elements. The situation gets more complicated, if the flag is not complete, as the following example will show.

\begin{example}\label{ex:fibersPhi}
Consider the case $V=\R^2$. As explained in Example \ref{example:calBbar(K^2)}, there are three types of points of $\calBbar(V)$:
\begin{enumerate}[(a)]
\item the homothety class $\eta$ of the constant norm,
\item classes of proper norms with $2$ different non-zero values,
\item classes of proper seminorms (\emph{i.e.}, seminorms with non-trivial kernel).
\end{enumerate}
If $K$ was algebraically closed, we would have $\calBbar_1(K)=\P^{1,\an}$ and these would exactly correspond to Berkovich type II, type III, respectively type I points of the space. However, $\R$ is not algebraically closed and thus $\calBbar_1(\R) \subsetneq \P^{1,\an}$.

As discussed before, if a point in $\calBbar_1(\R)$ is of type (b), then its fiber consists of two points. If the point is of type (c), then there is only one preimage, as the two choices of signed flags are equivalent. In contrast, the fiber of $\eta$ is quite large. For every line $V_1 \subset \R^2$, there are two equivalence classes of signed flags with underlying flag $0 \subsetneq V_1 \subsetneq \R^2$. Hence, there is a (non-canonical) bijection
$$\Phi^{-1}(\eta) \cong \P^1(\R) \sqcup \P^1(\R)\, .$$
Note that this bijection is not continuous, so it does not tell us anything about the topology. To find out more about the topology of this fiber we refer to \cite[Example 3.12]{JSY22}, since in the special case of dimension $2$, one can show that $\calBbar_1(\R)$ is the set of real points of $\P^{1,\an}$ and $\RX_1(\R)$ equals $\P^{1,\an}_r$. In Figure \ref{fig:Phi} we sketch the map $\Phi$. Note that because of the infinite nature of the spaces, the topology on both $\RX_1(\R)$ and $\calBbar_1(\R)$ is coarser than the figure might lead us to believe. For example, the space $\RX_1(\R)$ is connected (this follows from \ref{prop:tau_surjective}), even though as a set it is a disjoint union of closed intervals. 

\begin{figure}[h]
\begin{center}
\begin{tikzpicture}[scale=.5]

\draw[dotted](0,0) -- (0,7);
\draw[black] (0,3.5)--(-3,2.018)--(0,4);
\draw[black] (0,4.5)--(3,3.018)--(0,5);
\draw[black] (0,5.5)--(6,5.75)--(0,6);
\draw[black] (0,6.5)--(3,8.482)--(0,7);
\draw[black] (0,7.5)--(-6,7.75)--(0,8);




\draw[dotted] (0,0) ellipse (6 and 2);
\fill[black] (0,0) circle (2pt);

\draw[black] (3,1.732)--(-3,-1.732);
\draw[black] (-6,0)--(6,0);
\draw[black] (0,0)--(3,-1.732);

\draw[very thick, dotted] (-0.7,1.2) arc (90:150:1.5) ;

\end{tikzpicture}
\caption{The map $\Phi: \RX_1(\R)\to \calBbar_1(\R)$.}
\label{fig:Phi}
\end{center}
\end{figure}
\end{example}

\begin{example}
The computation of Example \ref{ex:fibersPhi} allows us to compute some fibers in the case where $K=\R\{\!\{t\}\!\}$ is the field of real Puiseux series. Consider the class $\eta$ of the Gauss norm $K^2\to \R,\ (x,y) \mapsto \max\{|x|,|y|\}$ in $\calXbar_1(K)$. We claim that just like in the case of real numbers with trivial valuation we have:
\[\Phi^{-1}(\eta) \cong \P^1(\R) \sqcup \P^1(\R) \, .\]
Let $\siinorm{\cdot}$ be a representative of a class in the preimage of $\Phi$ such that $|\siinorm{(x,y)}| = \max\{|x|,|y|\}$. Then surely, this signed norm is uniquely determined by its values on the unit sphere of $K^2$ with respect to the Gauss norm. For any element $u=(x,y) \in K^2$ with $\max\{|x|,|y|\}=1$, consider its reduction $\overline{u}=(\overline{x},\overline{y})\in \R^2$ modulo $t$, then by Lemma \ref{lem:rulesSignedSeminorms} (b), $\siinorm{u}= \siinorm{\overline{u}}$. Therefore, $\sgn(\siinorm{u})$ is uniquely determined by its reduction modulo $t$, i.e., by the restriction to $\R^2 \subset K^2$, and these restrictions are classified in Example \ref{ex:fibersPhi}. Vice versa, it is easy to see that any choice of restriction lifts uniquely to a signed seminorm on $K^2$.
\end{example}

\begin{proposition}\label{prop:tau_surjective}
The restriction map $\tau: \P_r^{n,\an} \xrightarrow{}  \RX_n(\R)$ is surjective.
\end{proposition}
\begin{proof}
Let $\big[\siinorm{\cdot}\big]$ be the class of a non-trivial signed seminorm on $(\R^{n+1})^{\ast}$. Then by Proposition \ref{prop:norm_diagonalizable} there is an ordered basis $B = (b_0,\dots,b_n)$ and $\Vec{c} \in \R^{n+1}$ with $c_0 \geq c_1 \geq \dots \geq c_n \in \R_{\geq 0}$ such that $\siinorm{\cdot} = \siinorm{\cdot}_{B,\Vec{c}}$. After scaling $\Vec{c}$, we can assume $c_n = 1$.

We define a signed multiplicative seminorm $\sinorm{\cdot}$ on $\R[t_0,\ldots,t_n]$ as follows. After a coordinate change we can write an element $f\in \R[t_0,\dots,t_n]$ as 
$$f=\sum_{I=(i_0,\dots,i_n)} a_I b_0^{i_0}\cdot \ldots\cdot b_n^{i_n}\, .$$
Using the lexicographic order $\geq_{\text{lex}}$ we define a monomial order as follows. For $I, J \in \N^n$ we define $I \geq_{\Vec{c}} J$ if the two conditions are satisfied:

\begin{itemize}
\item $c^{I} \geq c^J$  
\item $c^{I} = c^J$ implies $I \geq_{\text{lex}} J$.
\end{itemize}

Let $I_0$ be the leading monomial of $f$ with respect to this order and $a_{I_0} b^{I_0}$ the corresponding leading term. We define $|f|^{\sgn} = \sgn(a_{I_0}) c^{I_0}$. This is a signed multiplicative seminorm on $\R[t_0,\ldots,t_n]$. Indeed, if $f \in \R$, then $|{f}|^{\sgn} = \sgn(f) = \sgn(f) \cdot |f|_{\triv}$. Let $f, g \in \R[t_0,\ldots,t_n]$ and write $f=\sum_{I=(i_0,\dots,i_n)} a_I b_0^{i_0}\cdot \ldots\cdot b_n^{i_n}$ and $g =\sum_{I=(i_0,\dots,i_n)} a'_I b_0^{i_0}\cdot \ldots\cdot b_n^{i_n}$. Let $a_{I_0} b^{I_0}$ and $a'_{I_1} b^{I_1}$ be the corresponding leading terms. Since $\geq_{\Vec{c}}$ is a monomial order the leading term of $fg$ is $a_{I_0}a'_{I_0} b^{I_0 + I_1}$. This shows $|fg|^{\sgn} = |f|^{\sgn} |g|^{\sgn}$. We omit the verification of the strong triangle inequalities. Clearly, $\sinorm{\cdot}$ restricts to $\siinorm{\cdot}_{B,\Vec{c}}$ on $\R[t_0,\ldots,t_n]_1 \cong (\R^{n+1})^{\ast}$ and the homothety class of $\sinorm{\cdot}$ only depends on the homothety class of $\siinorm{\cdot}$. Thus, we have $\tau \big[\sinorm{\cdot}\big] = \big[\siinorm{\cdot}\big]$.
\end{proof}

\begin{remark}
With minor amendments, the above proof can be generalized to any non-Archimedean valued real closed field to show that each class of a diagonalizable signed seminorm has a preimage. In particular, the image of the natural map $\tau: \P_r^{n,\an}\to \RX_n(K)$ always contains the diagonalizable locus.
\end{remark}

\section{Real Bergman Fans}\label{section:realbergmanfans}

The tropical linear space associated to a matroid has the structure of a fan, called the \emph{Bergman fan}. Recall that to each oriented valuated matroid one can associate a real tropical linear space. If the valuation is trivial, \emph{i.e.}, we have an oriented matroid, the corresponding fan is called the \emph{real Bergman fan}. This construction in the case of finite ground sets was introduced in \cite[Chapter 2.3]{celaya-diss}. We will first discuss circuit axiomatization of infinite oriented matroids and then recall the notion of \emph{covectors} of an oriented matroid, which are then used to define its real Bergman fan. In this section, we consider $K$ to be a real closed and trivially valued field.

\subsection{Infinite Oriented Matroids}
A \emph{signed subset} is a function $C:E\to \{0,+1,-1\}$. We denote by $C^+\coloneqq \{e\in E \mid C(e)=+1\},\ C^-\coloneqq \{e\in E\mid C(e)=-1\}$, $\Supp(C)\coloneqq C^+\cup C^-$, and $C^0 \coloneqq E \setminus \Supp(C)$. The \emph{separation set} of two signed sets $C,C'$ is defined by $S(C,C')\coloneqq \{e\in E \mid C(e)=-C'(e)\neq 0\}$. 

We quickly discuss the circuit axiomatization of infinite oriented matroids of finite rank. Fix a possibly infinite ground set $E$. Recall that an oriented matroid 
is given by a chirotope (\emph{i.e.}, a Grassmann--Plücker function) $\varphi: E^n\to \mathbb{S}=\{0,+1,-1\}$. Equivalently, there is the same circuit axiomatization as in Proposition \ref{prop:circuit_description} (replacing $\R\T$ with $\mathbb{S}$), with one additional axiom, which ensures finite rank:
\begin{itemize} 
\item[\textbf{(C4)}] There is a positive integer $n \in \N$ such that for all subsets $A\subseteq E$ with $|A|>n$ there exists $C\in \mathcal{C}$ with $\Supp(C)\subseteq A$.
\end{itemize}

By restricting and applying the finite case (\emph{e.g.}, \cite[Theorem 3.2.5]{orientedMatroidsBook}) one obtains the \emph{strong circuit elimination}:
\begin{itemize}
\item[\textbf{(C3')}] For all $C\neq -C'\in \mathcal{C}$, $e \in S(C,C')$, and $f\in (C^+ \setminus C'^-)\cup (C^-\setminus C'^+)$, there is $C'' \in \mathcal{C}$ such that $C''(e)=0$, $f\in \Supp(C'')$, and
\begin{align*}
C''^+ \subseteq(C^+\cup C'^+), \quad C''^- \subseteq(C^-\cup C'^-).
\end{align*}
\end{itemize}
As a consequence, our notion of oriented matroids agrees with the characterization given in \cite[Theorem 3]{buchiFenton} with the additional requirement of finite rank.

\subsection{Covectors}
Let $\mathcal{M}$ be an oriented matroid given by a chirotope $\varphi:E^n\to \{0,+1,-1\}$.

\begin{definition}\label{def:cocircuit}
The set of \emph{signed cocircuits} $\mathcal{C}^*$ of $\mathcal{M}$ is given by the set of signed sets
\begin{align*}
\mathcal{C}^*=\{\pm[e\mapsto \varphi(\mu,e)] \mid \mu \in E^{n-1}\}\setminus \{0\}\, .
\end{align*}
Note that the choice of the sign of the chirotope $\varphi$ does not change the set of signed cocircuits.
\end{definition}

\begin{remark}
In the case that $E$ is finite, we have that $\mathcal{C}^*$ is a set of signed circuits of an oriented matroid, which is called the \emph{dual} oriented matroid. As expected, dualizing commutes with taking the underlying matroid. A crucial difference is, that the dual of an infinite (oriented) matroid is not of finite rank and hence behaves inherently different from our case. Hence, we will not consider any duals of infinite matroids. 
\end{remark}

The \emph{composition} of two signed sets $C,C':E\to \{0,+1,-1\}$ is defined by
$$(C\circ C')(e)=\begin{cases}
C(e) \quad \text{ if } C(e)\neq 0,\\
C'(e) \quad \text{if } C(e)= 0.\\
\end{cases}$$
We write $C\leq C'$ if $C'^0 \subseteq C^0$ and $C'|_{\Supp C}=C|_{\Supp C}$. In other words, this partial order is induced by the partial order $0 < +1,-1$, considered coordinatewise.

\begin{definition}\label{def:Covectors}
\ 
\begin{enumerate}[(a)]
\item A \emph{vector} of $\mathcal{M}$ is a finite composition of signed circuits.
\item Dually, a \emph{covector} is a finite composition of cocircuits. We consider the set of covectors $\Cov$ as a partially ordered set with the usual ordering of signed sets.
\end{enumerate}
\end{definition}

If $E$ is finite, the covectors of the oriented matroid are exactly the vectors of its dual \cite[Definition 3.7.1]{orientedMatroidsBook}.

\begin{example}
Let $\mathcal{M}$ be realizable and $\phi: K^{(E)}\to K^n$ be the corresponding realization. Then the set of vectors is given by
$$\{[e\mapsto \sgn(\lambda_e)] \mid (\lambda_e)_{e \in E} \in \ker \phi \}\, .$$

For the covectors, consider the dual map $\phi^*: K^n\to K^E$. Then the set of covectors is given by compositions of elements of
$$\{\sgn \circ \, \phi^{\ast}(v) \mid v \in K^n \}\, .$$
\end{example}

\begin{proposition}[Covectors of $\mathcal{M}_{\univ}$] \label{prop:covectorsMuniv}
For $K=\R$ and the vector space $(\R^{n+1})^*$,
the following sets of maps $(\R^{n+1})^*\to \{0,+1,-1\}$ are the same:
\begin{enumerate}[(a)]
\item The covectors of $\mathcal{M}_{\univ}$.
\item Maps of the form $X:(\R^{n+1})^* \to \{0,+1,-1\},\ e\mapsto \sgn(\siinorm{e})$ for some signed seminorm $\siinorm{\cdot}$ on $\R^{n+1}$.
\item Signed seminorms with range $\{0,+1,-1\}$.
\item All maps such that $X^+,X^-$ are strictly convex cones without $0$ and $\emptyset \neq X^0\subseteq(\R^{n+1})^*$ is a subspace.
\end{enumerate}
\end{proposition}
\begin{proof}
\noindent
We first show that the functions in (b) and (c) are exactly the covectors of $\mathcal{M}_{\univ}$.
\begin{enumerate}[(a)]
\setcounter{enumi}{1}
\item Let $X:(\R^{n+1})^* \to \{0,+1,-1\}$ be a cocircuit given by $\mu=(\mu_0,\dots,\mu_{n-1}) \in ((\R^{n+1})^*)^{n}$. Choose any $\mu_n\in (\R^{n+1})^*$ with $X(\mu_n)=+1$. Using multilinearity of the determinant, one can show 
$$X=\sgn(\siinorm{\cdot}_{B, \vec{c}})$$
for either for $B=(\mu_n,\mu_0,\dots,\mu_{n-1})$ or $B=(-\mu_n,\mu_0,\dots,\mu_{n-1})$ and $\vec{c}=(1,0,\dots,0)$.
Since the space of signed seminorms is closed under composition (\emph{cf.}\ Proposition \ref{prop:compSignedSeminorms}), we obtain (a) $\subseteq$ (b).

Vice versa, let $\siinorm{\cdot}$ be a signed seminorm. Then by Proposition \ref{prop:norm_diagonalizable} it is diagonalizable by an ordered basis $B=(b_0,\dots,b_n)$ with vector $\vec{c}=(c_0,\dots,c_n)$. Let $B_i=((-1)^ib_i,b_1,\dots,b_n$ and $\vec{c}_i\coloneqq (c_i,0,\dots,0,\dots,0)$. Then $\siinorm{\cdot}_{B,\vec{c}}=\siinorm{\cdot}_{B,\vec{c}_0} \circ \dots \circ \siinorm{\cdot}_{B,\vec{c}_n}$. As before we have that $\sgn(\siinorm{\cdot}_{B,\vec{c}_i})$ is a cocircuit, hence the composition is a covector.
\item Follows immediately from part (b) by replacing $\siinorm{\cdot}$ by the signed seminorm $\frac{\siinorm{\cdot}}{|\siinorm{\cdot}|}$.
\end{enumerate}
Now we show (c)=(d): From Lemma \ref{lem:convex_cone} we immediately deduce $\subseteq$.

For ``$\supseteq$'', let $X:(\R^{n+1})^*\to \{0,+1,-1\}$ be a function such that $X^+,X^-$ are strictly convex cones without $0$ and $\emptyset \neq X^0\subseteq(\R^{n+1})^*$ is a subspace. Strict convexity of the cones immediately implies axiom (i) from \ref{def:signedSeminorm}. Let $v,w \in (\R^{n+1})^*\setminus 0$. By convexity and (i), if $X(v)=X(w)$ or $X(v)=-X(w)$, then axiom (ii) is fulfilled. By symmetry, the only case that remains to show is if $X(v)=+1$ and $X(w)=0$. If $X(v+w)=-1$, then by convexity and (i) we have 
$$X(w)=X\left(\frac{w}{2}\right)=X\left(-\frac{v}{2}+\frac{v+w}{2}\right)=-1\, ,$$
which is a contradiction.
\end{proof}

\begin{remark}[Covector Axioms]\label{rem:CovectorAxioms}
The covector poset uniquely determines the oriented matroid \cite[\S 3.7]{orientedMatroidsBook}. If $E$ is finite, any poset of signed sets is the covector poset of an oriented matroid if and only if it fulfills:
\begin{itemize}
\addtolength{\itemindent}{.5cm}
\item[\textbf{(Cov1)}] $0\in \Cov$,
\item[\textbf{(Cov2)}] $X\in \Cov$ if and only if $-X\in \Cov$,
\item[\textbf{(Cov3)}] $\Cov$ is closed under composition,
\item[\textbf{(Cov4)}] For all $X,Y\in \Cov$ and $e\in S(X,Y)$ there exists $Z\in \Cov$ with $Z(e)=0$ and 
$$Z(e')=(X\circ Y)(e') \text{ for all } e'\not\in S(X,Y).$$
\end{itemize}
In other words, the covector axioms are a cryptomorphic definition of an oriented matroid. 
\end{remark}

If $E$ is infinite, the covector axiom \textbf{(Cov4)} need not be true. For example, one can show that the covector poset of the restriction of $\Muniv$ to $E'$ in Example \ref{ex:counterexampleCovectors} does not fulfill (iv). 

\begin{proposition}\label{prop:CovectorAxioms}
The covector poset $\Cov$ of $\mathcal{M}_{\univ}$ fulfills the covector axioms for $K=\R$.
\end{proposition}
\begin{proof}
Axioms \textbf{(Cov1)-(Cov3)} are obviously true, so it remains to show \textbf{(Cov4)}. 
Let $X,Y \in \Cov$ and $e\in S(X,Y)$ (in particular, neither $X=0$ nor $Y=0$). Without loss of generality, assume $X(e)=+1$. Set 
$$V\coloneqq  \Span(\{e\} \cup (X^0\cap Y^0)) \subseteq(\R^{n+1})^* \text{ and } \pi: (\R^{n+1})^* \to (\R^{n+1})^*/V$$ the projection map. We define 
$$A=(X\circ Y)^+ \setminus S(X,Y)=(X^+ \cap (Y^+\cup Y^0)) \cup (X^0 \cap Y^+)\, .$$
Then $(\R^{n+1})^*\setminus S(X,Y)= A \cup -A\cup (X^0\cap Y^0)$ and $Z\in \Cov$ fulfils axiom (iv) if and only if $Z|_V=0,\ Z|_A=+1,\ Z|_{-A}=-1$. We claim that $\pi(A)\cap \pi(-A)=\emptyset$. Consider $a\in A, \lambda \in \R, w\in X^0 \cap Y^0$; we have to show that $a+\lambda e +w \not \in -A$. If $\lambda=0$, then $X\circ Y (a+\lambda e+w)=X\circ Y(a)$. If $\lambda>0$, then
$$X(a+\lambda e +w)=X(a+\lambda e)=+1\, .$$
If $\lambda <0$, then $Y(\lambda e)=+1$ and thus
$$Y(a+\lambda e+w)=Y(a+\lambda e)+1\, .$$
In either case, $a+\lambda e+w \not \in -A$. By symmetry, this implies $\pi(A)\cap \pi(-A)=\emptyset$. Since $A$ is convex, so is $-A$ and hence also $\pi(A)$ and $\pi(-A)$. Therefore, we can apply the hyperplane separation Theorem \cite[3.3.9]{stoer2012convexity} in $(\R^{n+1})^*/V$ to the sets $\pi(A),\ \pi(-A)$ to obtain a separating hyperplane $H'\subseteq(\R^{n+1})^*/V$. The pull-back $H=\pi^{-1}(H')$ then separates $A,-A$ and $V\subseteq H$. By induction, on $H$ there is a covector fulfilling (iv), given by the sign of a diagonalizable seminorm, which in return is given by an ordered basis $(b_1,\dots,b_n)$ of $H$ and $(c_1\geq \dots \geq c_n)$. Choosing any $b_0\in A\setminus H$ and $c_0 >> c_i$ yields a signed seminorm $\siinorm{\cdot}$ on $(\R^{n+1})^*$ such that $Z=\sgn(\siinorm{\cdot})$ fulfills (iv).
\end{proof}

The \emph{restriction} $\mathcal{M}|_{E'}$ of an oriented matroid to a subset $E'\subseteq E$ is the oriented matroid whose signed circuits are the signed subsets of $\mathcal{C}$ whose support is contained in $E'$. One can easily check that the axioms \textbf{(C0)}-\textbf{(C4)} hold and that restriction commutes with taking the underlying matroid. This is equivalent to restricting a chirotope.

\begin{lemma}\label{lem:restrictionCovectors}
Let $\mathcal{M}$ be an oriented matroid on a (possibly infinite) ground set $E$ with covector set $\Cov$ and let $E'\subseteq E$ a finite subset. Then the restriction of $\mathcal{M}$ to $E'$ has covector set $\Cov'=\{X|_{E'} \mid X\in \Cov\}$. In particular, there is an order-preserving, surjective restriction map $\Cov \to \Cov'$.
\end{lemma}
\begin{proof}
If $E$ is finite, this follows from \cite[Proposition 3.7.11]{orientedMatroidsBook}. If $E$ be infinite, since composition of signed sets commutes with restriction, it suffices to consider cocircuits of $\mathcal{M}$ and show that their restriction to $E'$ is a covector. Choose a chirotope $\varphi$ and $X$ be a cocircuit given by $X(e)=\varphi(\mu_1,\dots , \mu_{n-1},e)$ for $\mu_1,\dots,\mu_{n-1} \in E$. Consider $E''=E'\cup \{\mu_1,\dots,\mu_{n-1}\}$. Then surely, $X|_{E''}$ is a covector (even a cocircuit), hence $X_{E'}=(X|_{E''})|_{E'}$ is a covector on $E'$ by the finite case.
\end{proof}

Consequences of the absence of duality leads to different behavior of the covector poset:

\begin{example}\label{ex:counterexampleCovectors}
The condition of $E'$ being finite in Lemma \ref{lem:restrictionCovectors} is indeed necessary. Consider the universal realizable matroid $\mathcal{M}_{\univ}$ with ground set $E=\R^3$. Let $E'=E\setminus H \cup \{v\}$, where $H\subset \R^3$ is any hyperplane and $v\in H\setminus 0$. Let $X$ be a cocircuit of $\mathcal{M}_{\univ}$ which is given by a basis $\mu$ of $H$ as in Definition \ref{def:cocircuit} with any chirotope. Then one can show that $X|_{E'}$ is not a covector of the restriction to $E'$.
\end{example}

\begin{proposition}\label{prop:covectors_to_flats}
Let $\mathcal{M}$ be an oriented matroid of rank $n$ with covector poset $\Cov$, and let $\mathcal{F}$ be the lattice of flats of the underlying matroid. The assignment
\begin{align*}
\Cov &\longrightarrow \mathcal{F} \\
X &\longmapsto X^0
\end{align*}
is well-defined, surjective, and strictly monotonic. In particular, any chain in $\Cov$ has length at most $n$.
\end{proposition}
\begin{proof}
Let $C^*\in \mathcal{C}^*$ be a cocircuit that is given by $\mu \in E^{n-1}$ as in Definition \ref{def:Covectors}. Then $(C^*)^0$ equals the closure of $\mu$ in the underlying matroid, hence it is a flat. Now the zero set of a covector equals the intersection of the zero sets of the circuits which appear in the composition, hence it is also a flat.

Let $F\in \mathcal{F}$ be any flat of rank $d\leq n$. Choose a basis $f_1,\dots,f_d,b_1,\dots,b_{n-d}$ of $M$ such that $F$ is the closure of $f_1,\dots f_d$. Define 
$$C_i^*(e)=\varphi(f_1,\dots,f_d,b_1,\dots,\hat{b}_i,\dots,b_{n-d}, e)\, ,$$
then for $X\coloneqq  C_1^* \circ \dots \circ C_{n-d}^*$ we have $X^0=F$.
\end{proof}

\subsection{Real Bergman Fans} 
Let $\mathcal{M}$ be an oriented matroid of rank $n+1$ on any ground set $E$ with poset of covectors $\Cov$. We define its \emph{real Bergman fan} $\Sigma^*_{\mathcal{M}}$ to be given by the collection of cones 
$$\langle X_1, \dots X_l\rangle_{\R_{> 0}} \subset \R\P^E$$
for each chain of non-zero covectors $X_1< \dots < X_l$. This is, as in the non-oriented case, a fan of pure dimension $n$. If $E$ is infinite, $\Sigma^*_{\mathcal{M}}$ may have infinitely many cones and is embedded into an infinite-dimensional real projective space.

\begin{remark}
\begin{enumerate}[(a)]
    \item In his thesis \cite{celaya-diss}, Celaya considered finite oriented matroids and showed that the real Bergman fan shares many of the same features of a Bergman fan. After taking the componentwise logarithm, the real Bergman fan restricted to the positive orthant coincides with the positive Bergman fan considered in \cite{ArdilaKlivansWilliams}.
    \item The real Bergman fan arises naturally from the perspective of matroids over hyperfields. The oriented matroid $\calM$ can be interpreted as a matroid
over the real tropical hyperfield $\R\T$, with trivial valuation. In \cite{anderson2019vectors}, the author introduces the notion of vectors of a matroid over a hyperfield where the $\R\T$- covectors are the signed covectors of the oriented matroid $\calM$. The set of $\R\T$-vectors of $\calM$ coincides exactly with
the support of $\Sigma^*_{\calM}.$ In unsigned tropical geometry, the Bergman fan of a matroid $M$ arises by considering it as a matroid over the Krasner hyperfield $\mathbb{K}$ and taking the set of $\mathbb{K}$-vectors of $M$.
\item In \cite{RRS22}, the authors showed that real phase structures on the Bergman fan of a finite matroid are in one to one correspondence with orientations of that matroid. In particular, the datum of the real Bergman fan is cryptomorphic to that of the Bergman fan and a choice of a real phase structure.
\end{enumerate}
\end{remark}

For finite oriented matroids, the notions of its real Bergman fan and its associated real tropical linear space agree:

\begin{proposition}[{\cite[Proposition 2.4.7, Corollary 2.4.8]{celaya-diss}}]\label{prop:BMfanIsTropM}
Let $\mathcal{M}$ be an oriented matroid on a finite ground set. Then
$$\Trop_r(\mathcal{M}) = \Sigma^*_{\mathcal{M}}$$
\end{proposition}

\begin{corollary}[{\cite[Proposition 2.5.4]{celaya-diss}}]\label{cor:BMfanIsTrop}
Let $\iota=(f_0:\dots:f_m): \P^n \xhookrightarrow{} \P^m$ be a linear embedding for any real closed and trivially valued field $K$. Let $\mathcal{M}_{\iota}$ be the associated realizable oriented matroid on $\{f_0,\dots,f_m\}\subset (K^{n+1})^*$. Then
\[\Trop_r(\P^n,\iota)=\Sigma^*_{\mathcal{M}_{\iota}}\, .\]
\end{corollary}

\begin{proposition}[{\emph{cf.} \cite[3.3.3]{celaya-diss}, \cite[Theorem 5.2.1]{orientedMatroidsBook}}] \label{prop:realtrophomeom}
The real Bergman fan $\Trop_r(\mathcal{M})$ of a finite oriented matroid $\mathcal{M}$ of rank $n$ is homeomorphic to $\R\P^{n}$.
\end{proposition}

\begin{theorem}\label{thm:realBMfanBuilding}
For $K=\R$ and the vector space $(\R^{n+1})^*$, there is a homeomorphism
$$\RX_{n}(\R) \cong \Sigma^*_{\Muniv}\, .$$
\end{theorem}
\begin{proof}
We can make the correspondence very explicit as follows. Let $x \in \Sigma^*_{\Muniv}$. Then there are $\alpha_i \geq 0$ and a flag of covectors $X_0 < \dots < X_n$ of $\Muniv$ such that $x = \sum_{i=0}^n \alpha_i e_{X_i}$. Hence there are signed cocircuits $Y_0\ldots,Y_n$ such that $X_0 = Y_0, X_1 = Y_0 \circ Y_1, \ldots, X_n = Y_0 \circ \dots \circ Y_n$. The signed cocircuits are given by $Y_i = \sgn(\siinorm{\cdot}_i)$ for signed seminorms $\siinorm{\cdot}_i$ on $(\R^{n+1})^*$. As in the proof of Proposition \ref{prop:covectorsMuniv} these signed seminorms are diagonalized by an ordered basis $B = (b_0,\ldots,b_n)$ and coordinate vectors $e_i = (0,\ldots,0,1,0,\ldots,0)$ such that $\siinorm{\cdot}_i = \siinorm{\cdot}_{B,e_i}$. Then, $x = (\alpha_1 + \dots + \alpha_n) e_{Y_1} \circ (\alpha_2 + \dots + \alpha_n) e_{Y_2} \circ \dots \circ \alpha_n e_{Y_n}$. Setting $c_0 = \alpha_0 + \dots + \alpha_n, \dots, c_n = \alpha_n$ we obtain $x = \siinorm{\cdot}_{B,\Vec{c}}$. Conversely, a signed seminorm $\siinorm{\cdot}_{B,\Vec{c}}$ yields coordinates $\alpha_i$ and a flag of covectors of $\Muniv$ by reversing the construction.
\end{proof}

\begin{remark} \label{remark:infinitebergmanfan}
\begin{enumerate}[(a)]
\item The interpretation of Remark \ref{rem:GIspaceIsTropMuniv} (b) becomes even more clear in the case of real numbers with trivial valuation. For any finite subset $\{f_0,\dots,f_m\}=E'\subset E = (\R^{n+1})^*$, there is a restriction map from the covector poset of $\Muniv$ to the one of $\Muniv|_{E'}$ by Lemma \ref{lem:restrictionCovectors}. The induced map of the respective order complexes is exactly the map $\RX_n(\R) \to \Trop_r(\P^n,\iota)$ for $\iota=(f_0:\dots:f_m)$.
\item Proposition \ref{prop:covectors_to_flats} gives us a map from the covectors poset of $\mathcal{M}_{\univ}$ to the lattice of flats of the underlying matroid, the universal realizable matroid. The continuous map on the corresponding order complexes is exactly the map $\Phi: \RX_n(\R) \to \calXbar_n(\R)$ induced by taking absolute values. 
\item Combining (a) and (b), we can reprove Theorem \ref{thm:commutativityPhi} for real numbers with trivial valuation. Let $\Cov$ be the covector lattice of $\mathcal{M}_{\univ}$, $\mathcal{F}$ the lattice of flats of the underlying matroid, and $E'\subset E$ a finite generating subset. We observe that the following diagram commutes:
\begin{center}
\begin{tikzcd}
\Cov \ar[r] \ar[d] & \Cov|_{E'} \ar[d]\\
\mathcal{F} \ar[r] & \mathcal{F}|_{E'} 
\end{tikzcd}
\end{center}
If now $\iota: \P^n\to \P(\R^{E'})$ is the corresponding embedding, we obtain commutativity for
\begin{center}
\begin{tikzcd}
\RX_n(\R) \ar[r] \ar[d] & \Trop_r(\P^n,\iota) \ar[d]\\
\calXbar_n(\R) \ar[r] & \Trop(\P^n,\iota) \, .
\end{tikzcd}
\end{center}
\item Following the proof of Theorem \ref{thm:realBMfanBuilding},
if we replace $\R$ by any other trivially valued real closed field $K$, the real Bergman fan agrees with the diagonalizable locus in $\RX_n(K)$. The existence of non-diagonalizable signed seminorms (\emph{cf.}\ \ref{ex:not_diagonalizable}) shows, that the real Bergman fan is then a proper subspace.
\end{enumerate}
\end{remark}

\begin{remark}[Signed Tropical Convexity]\label{rem:TropicalConvexity}
By \cite[Theorem 16]{yuyuster}, a tropicalized linear space is the tropical 
convex hull of its valuated cocircuits, and by \cite[Theorem 7.8]{loho2022signedtropicalhalfspacesandconvexity} 
the real tropicalization of a linear space is the TC-convex hull of its 
$\R\T$-cocircuits. Specializing to the universal matroid: a $n$-ordered 
subset $\mu=(\mu_1,\dots,\mu_n)$ of $E=(K^{n+1})^\ast$ defines a signed 
seminorm $\siinorm{\cdot}_\mu$ via $\siinorm{f}_\mu = \varphi_{\univ}(f,\mu_1,\dots,\mu_n)$. 
Any diagonalizable signed seminorm $\siinorm{\cdot}_{B,\Vec{c}}$ decomposes 
as a composition of scalar multiples of the $\siinorm{\cdot}_{\mu_i}$ with 
$\mu_i = (b_0,\ldots,\hat{b_i},\ldots,b_n)$. So in particular,
$\RX_n(\R) = \Trop_r(\Muniv)$ is the TC-convex hull of the signed valuated 
cocircuits.
\end{remark}

\subsection*{Outlook and applications}

Several lines of further inquiry naturally extend the framework of this paper.

\emph{Reductive groups over real closed fields.} The signed Goldman--Iwahori 
space provides a real-analytic, non-Archimedean analogue of a symmetric 
space for $\PGL$ over real closed fields. It should be a natural starting 
point for a signed analogue of Bruhat--Tits theory for reductive groups 
over real closed valued fields, paralleling the role of $\calXbar_n(K)$ 
in the non-Archimedean theory \cite{Werner_seminorms, RemyThuillierWernerI, 
RemyThuillierWerner_survey}.

\emph{The non-realizable case.} The limit theorem 
(Theorem~\ref{mainthm:limits}) ranges over linear embeddings of $\P^n$, equivalently, over realizable oriented valuated matroids, and 
Theorem~\ref{mainthm:GIspaceIsTropMuniv} identifies the limit with the 
real tropical linear space of the universal \emph{realizable} oriented 
valuated matroid. The construction of $\Trop_r(\mathcal{M})$ in 
Section~\ref{section:Muniv}, however, applies to arbitrary finite-rank 
oriented valuated matroids, and it is natural to ask what space arises 
as the inverse limit of $\Trop_r(\mathcal{M})$ over a suitable category 
of all (possibly non-realizable) such matroids. A description of this 
space would be interesting to study.

\emph{Real toric vector bundles.} A classical result of Klyachko, recast tropically by Kaveh and Manon \cite{km22}, classifies $T$-equivariant vector bundles on a toric variety $X_\Sigma$ in terms of piecewise  $\mathbb{Z}$-linear maps from the fan $\Sigma$ to the building $B(\PGL_n(K))$. A non-trivially valued generalization can be found in \cite{kaveh2022toric}.
Over a real closed field 
$K$, the signed Goldman--Iwahori space $\RX_n(K)$ is the natural 
replacement for the building, and it would be interesting to ask whether 
\emph{real} toric vector bundles on $X_\Sigma$ admit an analogous 
classification by piecewise linear maps $\Sigma \to \RX_n(K)$, with the signed data encoding the real structure beyond the underlying complex bundle.

\bibliographystyle{amsalpha}
\bibliography{biblio}{}

\end{document}